\documentclass[11pt,twoside]{article}
\usepackage{mathrsfs,amssymb,amsfonts,amsthm,amsmath,amsbsy}
\usepackage{natbib}
\usepackage{dsfont}
\usepackage{fancyhdr}
\usepackage[usenames,dvipsnames]{xcolor}
\usepackage[bookmarks,bookmarksnumbered,colorlinks=true,pdfstartview=FitV,
  linkcolor=Red,citecolor=blue!90!black,urlcolor=blue!90!black]{hyperref}
\usepackage{autonum}
\usepackage{cancel}

\allowdisplaybreaks

\bibpunct{\textcolor{blue!90!black}{(}}{\textcolor{blue!90!black}{)}}{\textcolor{blue!90!black}{;}}{a}{\textcolor{blue!90!black}{,}}{\textcolor{blue!90!black}{;}}

\newcommand{\bigtimes}{\raisebox{-.5mm}{$\text{\Large$\times$}$}}
\newcommand{\simplex}{\overline\nabla}
\newcommand{\simplexone}{\nabla}
\newcommand{\K}{\mathbb{K}}
\newcommand{\simplexH}{\overline{\K}_{H}}
\newcommand{\densesimplexH}{\K_{H}}
\newcommand{\simplexHo}{\K_{H}^{\circ}}
\newcommand{\T} {\mathbb{T}}

\newcommand{\ind}{\mathds{1}}

\newcommand{\bs}[1]{\boldsymbol{#1}}
\newcommand{\PD} {\mathrm{PD}}

\newcommand{\tn} {|\bs\theta|}

\newcommand{\Asv} {A}
\newcommand{\Amv} {A}
\newcommand{\A} {\mathcal{A}}
\newcommand{\AAA} {A}
\newcommand{\Abb} {\mathbb{A}}
\newcommand{\B} {B}
\newcommand{\BB} {\mathbb{B}}
\newcommand{\C} {{C}}
\newcommand{\D} {\mathcal{D}}

\newcommand{\N} {\mathbb{N}}
\newcommand{\PP} {\mathbb{P}}
\newcommand{\PPP} {\mathcal{P}}

\newcommand{\R} {\mathbb{R}}
\newcommand{\Z}{\mathbb{Z}}

\newcommand{\mm} {\mathbf{m}}

\newcommand{\Xt} {X}
\newcommand{\Xtau} {\mathcal{X}}

\numberwithin{equation}{section}
\newtheorem{theorem}{Theorem}[section]
\newtheorem{proposition}[theorem]{Proposition}
\newtheorem{corollary}[theorem]{Corollary}
\newtheorem{lemma}[theorem]{Lemma}
\newtheorem{example}[theorem]{Example}
\newtheorem{definition1}[theorem]{Definition}
\newenvironment{definition}{\begin{definition1}\rm}{\hfill$\square$\end{definition1}}
\newtheorem{remark1}[theorem]{Remark}
\newenvironment{remark}{\begin{remark1}\rm}{\hfill$\square$\end{remark1}}

\fancypagestyle{plain}{\fancyhead{}}

\title{\bf \vspace{-15mm}
Multiple Poisson--Dirichlet diffusions on\\ generalized Kingman simplices}
\author{Cristina Costantini$^{1}$ and Matteo Ruggiero$^{2}$\\[2mm]
\small $^{1}$National Group for Mathematical Analysis, Probability and Applications, INdAM, Italy\\
\small $^{2}$Stern School of Business, NYU Abu Dhabi}
\date{\today}

\begin{document}
\maketitle
\thispagestyle{empty}

\vspace{-5mm}
\begin{center}
\emph{This paper is dedicated to the memory of T.G.~Kurtz,
without whose work it simply would not exist.}
\end{center}

\begin{abstract}
We construct a new class of infinite-dimensional diffusions with values in a generalized Kingman simplex with finitely many marks. The model describes the temporal evolution of the relative frequencies of infinitely many types that are \emph{labeled} by a finite number $H$ of marks, but \emph{unlabeled} within each mark.
We first establish a \emph{blockwise skew-product} representation for a finite-type Wright--Fisher diffusion, extending the aggregation--renormalization self-similarity property of Dirichlet laws.
The decomposition separates an $H$-dimensional Wright--Fisher diffusion governing the evolving random mark masses, from $H$ Wright--Fisher diffusions, each run on its own random clock, which describe the evolution of the relative frequencies within each mark.
After ranking the within-mark frequencies in decreasing order, we identify the distributional limit as the number of types per mark tends to infinity and we derive an explicit form of its infinitesimal generator on a suitable domain.
The limiting diffusion admits the multiple Poisson--Dirichlet
distribution as a stationary distribution. It recovers the
infinitely-many-neutral-alleles diffusion when all types share
the same mark and yields a diffusion on the Thoma simplex
when there are two marks.
\end{abstract}

\noindent\textbf{Keywords}: infinitely-many-alleles model; multiparameter random time change; multiple random partition; skew-product representation; Thoma simplex; Wright--Fisher diffusion.

\medskip
\noindent\textbf{MSC 2020}: Primary 60J60; Secondary 60G10, 92D25.

\section{Introduction}

\subsection{Diffusion processes on infinite-dimensional simplices}
\label{sec:preliminaries}

The Poisson--Dirichlet diffusion, mostly known as \emph{infinitely-many-neutral-alleles model}, is a diffusion process taking values in the space
\begin{equation}
\begin{aligned}\label{simplex}
\simplex:=\bigg\{x\in [0,1]^{\infty}:\ x_1\ge x_2\ge
\ldots\ge 0,\sum_{i=1}^{\infty}x_i\le 1\bigg\},
\end{aligned}
\end{equation} 
sometimes called \emph{Kingman simplex}, which is the closure (in the product topology) of the infinite-dimensional ordered simplex
\begin{equation}
\begin{aligned}\label{simplexone}
\simplexone:=\bigg\{x\in [0,1]^{\infty}:\ x_1\ge x_2\ge\ldots\ge 0
,\sum_{i=1}^{\infty}x_i=1\bigg\}.
\end{aligned}
\end{equation} 
On a suitable domain, the generator of the Poisson--Dirichlet diffusion is defined on $\simplexone$ by
\begin{equation}
\begin{aligned}
\label{IMNAoperator}
\Abb^{\theta} f(x):=\frac{1}{2}\sum_{i,j=1}^{\infty}x_i(\delta_{ij}-x_j)\frac{\partial^2f}{\partial x_i\partial x_j}(x)
-\frac{\theta}{2}\sum_{i=1}^{\infty}x_i\frac{\partial f}{\partial x_i}(x),
\end{aligned}
\end{equation}
and on $\simplex\setminus\simplexone$ by continuous extension (which we denote by the same symbol).
Here $\theta>0$ and $\delta_{ij}$ is the Kronecker delta. 
\cite{EK81} showed that the closure of $\Abb^{\theta}$ generates a Feller semigroup on $C(\simplex)$ and the associated diffusion 
process is stationary and reversible with respect to Kingman's Poisson--Dirichlet 
distribution $\PD(\theta)$ \citep{K75}; cf.~also \cite{F10}, Theorem 
5.2. See \cite{EK93}, Section 9, and \cite{F10}, Section 5.1, for general reviews on this model.

The family of diffusions with operator $\Abb^{\theta}$ may be thought of as describing the temporal evolution 
of the decreasingly ranked frequencies of infinitely-many types in an ideally infinite population, subject to parent-independent mutation and random genetic drift. 
This interpretation is supported by constructions obtained through Wright--Fisher (WF) and Moran type discrete chains; see, for example, \cite{EK81}, Section 3. The model is also closely related to a class of Fleming--Viot 
probability-measure-valued diffusions \citep{FV79} driven by an ergodic pure-jump process, 
sometimes called the \emph{labeled} infinitely-many-neutral-alleles model; cf.~\cite{EK86}, Section 10.4.
In these models, one keeps track not only of the frequencies, but also of the types, by assigning to each frequency a \emph{label} (given by a location in an arbitrary Polish space). 
Ranking the atom masses of the Fleming--Viot process after forgetting the labels yields an unlabeled diffusion whose generator is $\Abb^{\theta}$; see \cite{EK93}, Theorem 9.2.1.

\cite{P09} introduced a two-parameter extension of the Poisson--Dirichlet diffusion, with generator
\begin{equation}\label{Petrov}
\Abb^{\alpha,\theta} f:=\Abb^{\theta} f-\frac{\alpha}{2}\sum_{i=1}^{\infty}\frac{\partial f}{\partial x_i}, \quad \quad 0\le \alpha<1, \quad \theta\ge-\alpha.
\end{equation} 
This diffusion is stationary and reversible with respect to  the two-parameter Poisson--Dirichlet distribution \citep{P95,PY97}, and was further investigated in \cite{FS10,Fetal11,RWF13,Cea17,
Fea21,
Fea22,Fea23,Gea24}. 

A natural extension of the Kingman simplex is the \emph{Thoma simplex}
\begin{equation}\label{Thomasimplex}
\T:=\bigg\{(x,y):\,x,y\in\simplex,\ \sum_{i=1}^{\infty}(x_i+y_i)\leq 1\bigg\},
\end{equation}
where $\{x_i\}_{i\geq 1}$ and $\{y_i\}_{i\geq 1}$ may be viewed as the frequencies of infinitely-many types labeled with two possible marks, separately ranked for each mark. 

Diffusions on $\T$ were introduced in \cite{BO09} and 
\cite{O10}. In the latter work, the process degenerates to the Poisson--Dirichlet diffusion in a suitable limit; see also \cite{O18,K20,K24}.

In this paper we aim at constructing a diffusion model for the temporal evolution of infinitely-many type frequencies that are \emph{labeled} by a finite number of marks, but \emph{unlabeled} within each mark.
A natural state space for this process is the \emph{generalized Kingman simplex}
\begin{equation}
\begin{aligned}
\label{extendedKingmansimplex}
\simplexH:=\bigg\{z=(z_1,\ldots ,z_H):\,z_h\in\simplex,\,\sum_{h=
1}^H\sum_{i=1}^{\infty}z_{h,i}\leq 1\bigg\},
\end{aligned}
\end{equation}
which is the closure of 
\begin{equation}
\begin{aligned}
\label{denseKingmansimplex}
\densesimplexH:=\bigg\{z=(z_1,\ldots ,z_H):\,z_h\in\simplex,\,\sum_{h=
1}^H\sum_{i=1}^{\infty}z_{h,i}= 1\bigg\},
\end{aligned}
\end{equation}
where each $z_h\in\simplex$ describes the ranked frequencies of infinitely-many types with mark $h$, for $h=1,\ldots,H$. Clearly, $\simplexH$ reduces to \eqref{simplex} when $H=1$ and to \eqref{Thomasimplex} when $H=2$.

The diffusion process we construct on $\densesimplexH$ admits as a stationary distribution the recently introduced \emph{multiple Poisson--Dirichlet distribution} \citep{S24a,S24b} (see Definition~\ref{def:mPD} below), which generalizes Kingman's Poisson--Dirichlet distribution to a probability measure on $\densesimplexH$. Accordingly, we call our process a \emph{multiple Poisson--Dirichlet diffusion}. When $H=1$, it reduces to the Poisson--Dirichlet diffusion with generator \eqref{IMNAoperator}, while for $H=2$ it takes values in the Thoma simplex but differs from the diffusions studied in \cite{BO09,O10} (see Remark~\ref{rem:comparison-BO09}).

The multiple Poisson--Dirichlet distribution can be characterized
as the unique probability measure on $\densesimplexH$ corresponding
to Strahov's multiple partition structure; see \cite{S24a,S24b}.
This parallels the classical characterization of the
Poisson--Dirichlet distribution as the unique probability
measure on $\simplexone$ associated with Ewens' partition structure;
see \cite{K78a,K78b}.
Strahov's representation is based on coherent systems and
harmonic functions on a branching graph of multiple partitions.
In \cite{S24b}, Strahov also considers a Wright--Fisher-type
population model with infinitely many alleles classified
into finitely many mutation classes and derives the refined
Ewens sampling formula.
Our contribution is to construct a continuous-time diffusion 
that describes the temporal evolution of the ranked frequencies within each class (each mark, in our terminology) and to
determine the operator that characterizes it as the unique solution of a martingale problem.

In the one-mark case, \cite{P09} constructs the
Poisson--Dirichlet diffusion, and more generally the
two-parameter Poisson--Dirichlet diffusion, as a limit
of Markov chains on partitions.
By contrast, \cite{EK81} constructs the one-parameter
Poisson--Dirichlet diffusion as a limit of Wright--Fisher
(WF) diffusions after ranking their components in decreasing
order. A WF-based construction of the two-parameter diffusion
is given in \cite{Cea17}.
Here we follow the WF diffusion route and do not rely on
the representation-theoretic framework of \cite{S24a,S24b}.
Our convergence and identification arguments nevertheless
differ from those of \cite{EK81}. 
In particular, in our case the limiting operator cannot be
viewed as an operator on $\C(\simplexH)$, unlike the
corresponding operators in the one-mark setting
\cite{EK81,P09,Cea17}
(see Subsection~\ref{subsec:domain-subtlety}).

Our starting point is a WF diffusion with $HK$ types, grouped into $H$ marks. The main tool is a \emph{blockwise skew-product representation} of this diffusion: at the upper level, the mark masses evolve as an $H$-type WF diffusion; at the lower level, conditional on the mark masses, the within-mark relative frequency vectors evolve as independent WF diffusions, each run on a random clock determined by the corresponding mark mass (Section~\ref{sec:hierarchicalWF}). 
Related skew-product constructions in other contexts include, for example, \cite{G63, WY98, P11, KS88, P91, Fea23, 
FPRW18}. In those constructions, typically there is a single random clock (in the classical case of \cite{G63}, determined by the random length  of the process). Here we have a different random clock for each mark (determined by the random mark mass), leading to a \emph{multiparameter random time change} (see \cite{Hel74, Kur80}, \cite{EK86}). 

We work under the assumption $\theta_h\ge 1$ for all $h$. This restriction ensures that the random clocks are well defined; see \eqref{eq:rtc} and Lemma \ref{lem:thetas}. 
The case when the parameters $\theta_h$ are allowed to be smaller than $1$ presents significant additional technical difficulties and we leave it for future work.

We first obtain the limiting diffusion in this hierarchical representation and then identify its infinitesimal generator in the original coordinates (Section~\ref{sec:convergence}).

The state space $\densesimplexH$ is natural in population genetics when alleles are classified into finitely many functional classes. A prominent example is cystic fibrosis, caused by mutations in the \emph{cystic fibrosis transmembrane conductance regulator} gene \citep{CT09}: the roughly 2000 alterations described to date cause a wide range of disease severity and can be grouped into six or seven classes; see, e.g., \cite{Rea05,DA16}. Such classification is clinically relevant because available treatments are often class-specific; see \cite{Vea16}. Our construction thus provides a model for the temporal dynamics both of the frequency of each class and of the relative allele frequencies within each class. 

Another potential application is Bayesian inference. A large literature studies time evolving random measures with Dirichlet, one- or two-parameter Poisson--Dirichlet and related distributions as equilibrium measures; see the recent review \cite{Qea22}. The induced laws on path space can be used as priors for temporally correlated data, enabling posterior inference on trajectories and parameters, typically via MCMC or sequential Monte Carlo; see, e.g., \cite{Car17,MR16,RT08,KKK21}. 
In this perspective, our diffusion provides a prior on continuous trajectories with values in $\simplexH$ for data drawn from an infinite pool of types partitioned into finitely many classes. Moreover, \cite{PRS16,Aea21,Aea23} establish analytical tractability for certain hidden Markov models driven by Fleming--Viot/Dawson--Watanabe diffusions, under suitable observation schemes, and \cite{DPea25} for similar models driven by Petrov's diffusion \eqref{Petrov}. We expect the present model to enjoy similar tractability properties (not pursued here), potentially allowing closed-form calculations in a Bayesian posterior analysis.

%

\subsection{Overview of the results}\label{sec:overview}

We provide an informal overview of the construction developed in the rest of the paper.
Consider an ideally infinite population with finitely many types, where types are \emph{marked}. We assume
there are $H\ge 1$ marks and $K\ge 2$ types for each mark; we denote by $(h,i)$ the $i$th type among those
with mark $h$ ($i=1,\ldots,K$, $h=1,\ldots,H$), so the total number of types is $HK$.
We refrain from using the terminology ``subgroup $h$'' to avoid suggesting a spatial structure, which we do not consider here. 

For $L\in \N$, let $\Delta_L$ denote the usual $L$ probability simplex
\begin{equation}\Delta_L:=\bigg\{z\in [0,1]^L:\,\sum_{h=1}^L z_h=1\bigg
\}.
\label{findimsimplex}
\end{equation}
We let the vector of type frequencies $Z^K$, with states $z=(z_1,\ldots,z_H)\in\Delta_{HK}$ with $z_h=(z_{h,1},\ldots,z_{h,K})$, evolve as an $HK$-type WF diffusion with generator
\begin{equation}
\begin{aligned}\label{B^K}
\B^K f(z)
:=&\,\frac{1}{2}\sum_{h,k=1}^H\sum_{i,j=1}^K
z_{h,i}(\delta_{h,k}\delta_{i,j}-z_{k,j})
\frac{\partial^2f}{
\partial z_{h,i}\partial z_{k,j}}
+\frac{1}{2}\sum_{h=1}^H\sum_{i=1}^K
\bigg(\frac{\theta_h}{K}-\tn z_{h,i}\bigg)\frac{\partial f}{\partial z_{h,i}},
\end{aligned}
\end{equation}
acting on $\D(\B^K):=\C^2(\Delta_{HK})$, where $\theta_1,\ldots,\theta_H>0$,
$\bs\theta:=(\theta_1,\ldots,\theta_H)$ and $\tn:=\sum_{h=1}^H\theta_h$. Here and later we define, for a closed set $D\subset \R^d$, the space of twice-differentiable functions on $D$ as
\begin{equation}
\label{C2convention}
\C^2(D):=\{f\in \C(D):\ \exists\tilde {f}\in \C^2(\R^d),\ \tilde {f}
\vert_D=f\},
\end{equation} 
where $\C(D)$ is the set of continuous functions on $D$. 

In the classical scaling-limit construction from WF Markov chains (cf., e.g., \cite{EK86}, Section~10.2; \cite{E09}, Section~4.1),
the second-order term corresponds to mean-field resampling (genetic drift), while the first-order term corresponds to parent-independent mutation
to type $i$ within mark $h$ at rate $\theta_h/(2K)$ (for simplicity we allow self-mutations). We assume absence of selection throughout.

In view of the blockwise skew-product decomposition below, we work under the standing assumption that $\theta_h\ge 1$ for all $h$, which ensures that the multiparameter random time change \eqref{eq:rtc} is well defined (see Lemma \ref{lem:thetas}).

When $H=1$, \cite{EK81} showed that ranking the $K$ frequencies in decreasing order and letting $K\to\infty$ yields the Poisson--Dirichlet diffusion with generator \eqref{IMNAoperator}.
Here we want to carry out an analogous operation for fixed $H>1$: for each mark $h$, we want to rank the within-mark  frequencies $\{Z^K_{h,i}\}_{i=1}^K$ in decreasing order and let $K\to\infty$.

A convenient approach is suggested by the aggregation--renormalization (self-similarity) property of Dirichlet laws, recalled in Proposition~\ref{pro:Dirselfsimilarity}:
grouping Dirichlet-distributed frequencies according to a partition of the indices, both the sums for each group and the 
relative frequencies within each group are still Dirichlet distributed with appropriate parameters. 

Accordingly, writing $Z^K:=(Z_1^K,\ldots,Z_H^K)$ with $Z_h^K:=(Z_{h,1}^K,\ldots,Z_{h,K}^K)$, we consider the mark masses process
\[
\Bigg(\sum_{i=1}^K Z_{1,i}^K(\cdot),\ldots,\sum_{i=1}^K Z_{H,i}^K(\cdot)\Bigg)
\]
and, for each $h$, the within-mark relative frequency process
\begin{equation}\label{normalization in overview}
\Bigg(\frac {Z_{h,1}^K(\cdot )}{\sum_{i=1}^KZ_{h,i}^K(\cdot )},\ldots 
,\frac {Z_{h,K}^K(\cdot )}{\sum_{i=1}^KZ_{h,i}^K(\cdot )}\Bigg).
\end{equation}
It turns out that, formally, the two above processes have joint generator of the form 
\begin{align}
\A^Kf(w,x)
=&\,\frac{1}{2}\sum_{h,k=1}^Hw_h(\delta_{h,k}-w_
k)\frac {\partial^2f}{\partial w_h\partial w_k}+\frac{1}{2}\sum_{h=1}^
H\Big(\theta_h(1-w_h)-(\tn-\theta_h)w_h\Big)\frac {\partial f}{\partial 
w_h}\nonumber\\
&\,+\sum_{h=1}^H\frac{1}{w_h}\,\bigg[\frac{1}{2}\sum_{i,j=1}^Kx_{h,i}(\delta_{
i,j}-x_{h,j})\frac {\partial^2f}{\partial x_{h,i}\partial x_{h,j}}+\frac {\theta_
h}2\sum_{i=1}^K\bigg(\frac 1K-x_{h,i}\bigg)\frac {\partial f}{\partial 
x_{h,i}}\bigg],
\end{align}
where $w_h$ is the mass of mark $h$ and $x_{h,i}$ is the relative frequency of type $i$ within mark $h$.
The first line is the generator of an $H$-type WF diffusion, with parent-independent mutation, while each bracketed term is the generator of a $K$-type WF diffusion, with symmetric mutation. 
The coefficients $1/w_h$ can be interpreted as deriving from a
{\em multiparameter random time change}: this leads to the hierarchical representation proved in Section~\ref{sec:hierarchicalWF}, which can be informally described as follows. 
For $\theta_h\geq 1$ for all $h$, if 
$\sum_{i=1}^KZ_{h,i}^K(0)>0$ for all $h=1,\ldots ,H$, 
and  
\begin{equation}\nonumber
\bigg(\sum_{i=1}^KZ_{1,i}^K(0)\, ,\ldots,\,\sum_{i=1}^KZ_{H,i}^K(0)\bigg)\,,\,\frac{Z_1^K(0)}{\sum_{i=1}^KZ_{1,i}^K(0)}\,,\ldots ,\,\frac{Z_H^K(0)}{\sum_{i=1}^KZ_{H,i}^K(0)}
\end{equation} 
are mutually independent, 
it holds that
\begin{equation}\label{eq:rz}\begin{aligned}
Z^K:=\big(Z^K_1,\ldots ,Z^K_H\big)\stackrel d{=}\big(W^K_1\mathcal{X}_
1^K,\ldots ,W^K_H\mathcal{X}_H^K\big),\end{aligned}
\end{equation}
with $W^K$ an $H$-type WF diffusion with mutation 
parameters $\theta_1,\ldots ,\theta_H$, and $\mathcal{X}_h^K$ defined as
\begin{equation}
\mathcal{X}_h^K(t):=\Xt_h^K\bigg(\int_0^t\frac 1{
W_h^K(s)}\,d {s}\bigg),\qquad h=1,\ldots ,H,\label{eq:rtc}
\end{equation}
where $W^K,X_1^K,\ldots ,X_H^K$ are independent and $X_h^K$ is a $K$-type WF diffusion with  symmetric mutation with rate $\theta_h/(2K)$. 
Eq.~\eqref{eq:rtc} expresses the fact that changes in 
the relative frequencies of the types with mark $h$ occur 
on the time scale $\int_0^t\big(1/W^K_h(s)\big)d {s}$. 

The representation \eqref{eq:rz}-\eqref{eq:rtc} allows us to prove, under a mild assumption on the initial distributions, that the process obtained from $Z^K$ by ranking the within-mark frequencies converges in distribution, as $K\to\infty$, to a Markov process $Z$.
While one may conjecture from \eqref{B^K} that the limiting generator has the form
\begin{equation}\label{eq:Bfalse}
\hat{\BB}f(z_1,\ldots,z_H)
:=\frac{1}{2}\sum_{h,k=1}^H\sum_{i,j=1}^{\infty}
z_{h,i}(\delta_{h,k}\delta_{i,j}-z_{k,j})\frac {\partial^2f}{\partial 
z_{h,i}\partial z_{k,j}}-\frac {\tn}2\sum_{h=1}^H\sum_{i=1}^{\infty}
z_{h,i}\frac {\partial f}{\partial z_{h,i}},
\end{equation}
it turns out that $Z$ is characterized as the unique solution of a martingale problem and the operator of the martingale problem acts as \eqref{eq:Bfalse} only on a subclass of functions in its domain.
In fact, the functions in the domain of this operator can be written as
\begin{equation}\label{f(z)=f_0f_1}f(z):=f_0(|z_1|,\ldots ,|z_H|)
F(z),\quad\quad |z_h|:=\sum_{i\ge 1}z_{h,i},\end{equation}
where $f_{0}$ and  $F(z)$ have a specific form (see \eqref{Phi grande}-\eqref{final D(B)})
and the operator of the martingale problem is given by 
\begin{equation}\label{eq:B}\begin{aligned}
\BB f=\hat{\BB }f+\frac 12\sum_{h=1}^H\theta_h\frac {\partial f_0}{
\partial |z_h|}F.
\end{aligned}
\end{equation}
All this will be fully detailed in Section \ref{sec:convergence}.

Note that we could isolate the mark-specific dynamics 
and highlight the interaction among the marks by writing, for $f$ as in \eqref{f(z)=f_0f_1},
\begin{equation}\label{B as indep dynamics + interaction}
\BB f=\sum_{h=1}^H\BB_hf-\frac{1}{2}
\sum_{1\le k\ne h\le H}\sum_{i,j=1}^{\infty}z_{h,i}z_{k,j}\frac {\partial^2 
f}{\partial z_{h,i}\partial z_{k,j}},
\end{equation}
where
\begin{equation}\label{B_h}
\BB_hf:=\frac{1}{2}\sum_{i,j=1}^{\infty}z_{h,i}(\delta_{ij}-z_{h,j})\frac {\partial^2f}{\partial z_{h,i}\partial 
z_{h,j}}-\frac {\tn}2\sum_{i=1}^{\infty}z_{h,i}\frac {\partial f}{
\partial z_{h,i}}+\frac {\theta_h}2\frac {\partial f_0}{\partial 
|z_h|}F.
\end{equation}

Section~\ref{sec:ALTstationarity} shows the connection between the process constructed in Section \ref{sec:convergence} and the \emph{multiple Poisson--Dirichlet distribution} recently introduced by \cite{S24a,S24b}. 
Recall that Kingman's $\PD(\theta)$ distribution on $\simplexone$ is 
the de~Finetti measure in Kingman's representation theorem for the Ewens partition structure (see \cite{K78a,TE97,C16,T21}) and, at the same time, the reversible law of the Poisson--Dirichlet diffusion \eqref{IMNAoperator}.
{\cite{S24a,S24b} constructs a \emph{multiple partition structure} analogous to the Ewens partition structure, but with the elements of the partitions labeled by finitely many marks, and then generalizes Kingman's representation theorem.}
The corresponding de~Finetti measure is the probability law on $\densesimplexH$ recalled next. 
 
\begin{definition}(Multiple Poisson--Dirichlet distribution)\label{def:mPD}\ \\
Let $\upsilon=(\upsilon_1,\ldots,\upsilon_H)\sim\mathrm{Dir}_H(\theta_1,\ldots,\theta_H)$ and, independently, let $\xi_h\sim\PD(\theta_h)$ for $h=1,\ldots,H$. Then $\zeta=(\zeta_1,\ldots,\zeta_H)\in\densesimplexH$, where $\zeta_h:=\upsilon_h\xi_h=(\upsilon_h\xi_{h,1},\upsilon_h\xi_{h,2},\ldots)$, is said to have \emph{multiple Poisson--Dirichlet} distribution $\PD(\theta_1,\ldots,\theta_H)$.
\end{definition}

\begin{remark}\label{re:StraMPD}
In \cite{S24b}, the multiple Poisson--Dirichlet distribution is actually the joint law of $(\zeta,\upsilon)$, but clearly the two formulations are equivalent. 
\end{remark}

It turns out that the multiple Poisson--Dirichlet distribution in Definition~\ref{def:mPD} is a stationary distribution for the process $Z$ constructed in Section~\ref{sec:convergence}. For this reason, we call $Z$ the \emph{multiple Poisson--Dirichlet diffusion}.

Finally, we provide a construction of 
the multiple Poisson--Dirichlet distribution 
analogous to Kingman's construction of $\PD(\theta )$ as the limit in distribution of 
ranked symmetric Dirichlet random frequencies. 

%



\subsection{Structure of the paper}\label{subsec:structure}

In Section~\ref{sec:hierarchicalWF} we prove a blockwise skew-product representation for the $HK$-type WF diffusion: we first recall the aggregation--renormalization property of the $HK$- type Dirichlet law, then introduce an appropriate multiparameter time-changed dynamics (Subsection \ref{chartchang}), and finally show that, under aggregation--renormalization into $H$ blocks of $K$ components each, the $HK$-type WF diffusion follows such dynamics (Theorem~\ref{th:Krepr}). 

In Section~\ref{sec:convergence}, we rank the coordinates of the WF diffusion within each block and let $K\to\infty$, yielding the multiple Poisson--Dirichlet diffusion (Theorem \ref{Z1,Z2aslimit}); we also 
characterize the multiple Poisson--Dirichlet diffusion as the unique solution of a martingale problem (Theorem \ref{explicit B and domain}) and discuss why, in this setting, the operator of the martingale problem cannot be taken as an operator on 
$\C(\simplexH)$, as was the case, for  $H=1$, in \cite{EK81,P09, Cea17}. 

Finally, in Section~\ref{sec:ALTstationarity} we show that the process constructed in Section \ref{sec:convergence} is stationary with respect to the multiple Poisson--Dirichlet distribution (Corollary \ref{th:ALTBstar-stat}) and that, consistently, the multiple Poisson--Dirichlet distribution can be recovered as the limit of the distributions of  blockwise ranked Dirichlet random frequencies (Theorem \ref{th:MPDappr}).


\section{Blockwise skew-product representation for WF diffusions}\label{sec:hierarchicalWF}

In this section, after recalling the aggregation--renormalization property of Dirichlet laws (Section \ref{selfsim}), we construct the random time-changed process that is the candidate for the blockwise skew-product form of an $HK$-type WF diffusion and we characterize it as the unique solution of a martingale problem (Section \ref{chartchang}). This characterization allows us to derive the desired blockwise skew-product decomposition in Section \ref{skprodform}.


\subsection{Dirichlet self-similarity under aggregation and renormalization}\label{selfsim}
We start by recalling a well known self-similarity property for Dirichlet distributions, in a parameterization tailored to our purposes. Recall that a Dirichlet distribution $\mathrm{Dir}_L(a_1,\ldots ,a_L)$ on $\Delta_{L}$ as in \eqref{findimsimplex} has density
with respect to the $(L-1)$-dimensional Lebesgue measure, given by 
\begin{equation}\label{eq:Dir}
\frac{\Gamma(\sum_{l=1}^{L}a_{l})}{\prod_{l=1}^{L}\Gamma(a_{l})}w_{1}^{a_{1}-1}\cdots w_{L}^{a_{L}-1}, \quad a_{1},\ldots,a_{L}>0,
\end{equation} 
where $w_{L}=1-\sum_{l=1}^{L-1}w_{l}$, and $\Gamma(b)=\int_{0}^{\infty}y^{b-1}e^{-y}d y$.

\begin{proposition}\label{pro:Dirselfsimilarity}
Let 
\begin{equation}\nonumber
(\zeta_1,\ldots ,\zeta_{HK})\sim\mathrm{Dir}_{HK}\bigg(\frac{\theta_
1}K,\ldots ,\frac{\theta_1}K,\frac{\theta_2}K,\ldots ,\frac{\theta_
2}K,\ldots ,\frac{\theta_H}K,\ldots ,\frac{\theta_H}K\bigg),
\end{equation}
 where each $\theta_h$ is repeated $K$ times, and 
define  $\upsilon :=(\upsilon_1,\ldots ,\upsilon_H)$ and 
$\xi_h:=(\xi_{h,1},\ldots ,\xi_{h,K})$ by 
\begin{equation}\nonumber
\upsilon_
h:=\sum_{j=(h-1)K+1}^{hK}\zeta_j, \quad \quad 
\xi_{h,i}:=
\frac{\zeta_{(h-1)K+i}}{
\sum_{j=(h-1)K+1}^{hK}\zeta_j},
\end{equation}
for $h=1,\ldots ,H$ and $i=1,\ldots ,K$. Then
\begin{equation}
\begin{aligned}
\label{eq:upsilonxi}
\upsilon\sim\mathrm{Dir}_H(\theta_1,\ldots ,\theta_H),\quad \quad 
\xi_h\sim\mathrm{Dir}_K\bigg(\frac{\theta_h}K,\ldots ,\frac{\theta_
h}K\bigg), \quad h=1,\ldots ,H,
\end{aligned}
\end{equation} 
where $\upsilon ,\xi_1,\ldots ,\xi_H$ are independent. The converse is also true, namely, given $\upsilon$ and $\xi_h$, $h=1,\ldots ,H$ as in \eqref{eq:upsilonxi}, the vector $(\upsilon_1\xi_{1,1},\ldots ,\upsilon_1\xi_{1,K},\ldots ,\upsilon_H\xi_{H,1},\ldots ,\upsilon_H\xi_{H,K})$ has law
\begin{equation}\nonumber
\mathrm{Dir}_{HK}\bigg(\frac{\theta_
1}K,\ldots ,\frac{\theta_1}K,\frac{\theta_2}K,\ldots ,\frac{\theta_
2}K,\ldots ,\frac{\theta_H}K,\ldots ,\frac{\theta_H}K\bigg).
\end{equation}
\end{proposition}
\begin{proof}
This is a special case of a more general result which relies on the construction of Dirichlet distributions through normalization of independent gamma random variables. See Theorem 1.3.1 in \cite{F10} and Proposition G.3 in \cite{GvdV17}.
\end{proof}

The above self-similarity holds for general 
partitions of the indices in the vector of parameters in \eqref{eq:Dir}, 
and can be also extended to gamma subordinators and Dirichlet 
random probability measures. See, e.g., \cite{F10}, Theorem 2.23, and \cite{GvdV17}, Section 4.1.2.


\subsection{Construction and characterization of the random time-changed process}\label{chartchang}

Our goal is to provide an analog of Proposition \ref{pro:Dirselfsimilarity} for WF diffusions. In this analog,
$\zeta_1,\ldots ,\zeta_{HK}$ 
are replaced by an $HK$-type WF 
diffusion $(Z_1,\ldots ,Z_{HK})$ with parent-independent 
mutation occurring at rates  
$$\frac{\theta_{1}}{2K},\ldots,\frac{\theta_{1}}{2K},\frac{\theta_{2}}{2K},\ldots,\frac{\theta_{2}}{2K},\ldots,\frac{\theta_{H}}{2K},\ldots,\frac{\theta_{H}}{2K},$$
and $\upsilon ,\xi_1,\ldots ,\xi_H$ by $W^K,\Xtau_1^
K,\ldots ,\Xtau_H^K$, 
where $W^K$ is an $H$-type WF diffusion with parent-independent mutation with rates $\theta_1/2,\ldots ,\theta_H/2$ (depending on $K$ only by its initial distribution) 
and each $\Xtau_h^K$ is a 
$K$-type WF diffusion, with symmetric mutation with rates $\theta_h/(2K)$, observed on a random time scale, 
namely  
\[
\Xtau_h^K(t):=X_h^K\bigg(\int_0^t\frac{1}{W^K_
h(s)}ds\bigg).\] 
Informally, the $HK$ types are marked with $H$ labels, each marking $K$ types;  
$W^K$ drives the evolution of the masses assigned 
to each group of types with the same mark, while $\Xtau^K_h$ describes the 
evolution of the relative frequencies within the group of types with mark $h$. 

The number of marks  $H$ 
will be fixed throughout the paper, hence typically 
dropped as a superscript or subscript unless needed for clarity, 
whereas the number of types with the same mark, $K$, is kept fixed in this section but later will be let diverge to infinity.

More precisely, let $W^K=(W^K_1,\ldots ,W^K_H)$ be a WF diffusion on the $H$-simplex $\Delta_H$ (cf. \eqref{findimsimplex})
with generator, for $\bs\theta=(\theta_{1},\ldots,\theta_{H})$ and $\tn  :=\sum_{h=1}^H\theta_h$,
\begin{equation}
\begin{aligned}
\label{eq:hatA0}
\Asv_0f(w):=&\,\frac{1}{2}\sum_{h,k=1}^Hw_h(\delta_{h,k}-w_k)\frac{
\partial^2f}{\partial w_h\partial w_k}
+\frac{1}{2}\sum_{h=1}^H
(\theta_{h}-\tn  w_{h})
\frac{\partial f}{
\partial w_h},
\end{aligned}
\end{equation} 
acting on the domain $\D
(\Asv_0):=\C^2(\Delta_H)$ (cf. \eqref{C2convention}), where $\delta_{h,k}$ is the Kronecker delta. See \cite{EK86}, Theorem 8.2.8. 
By Theorem 4.4.1 of \cite{EK86}, $W^K$ is the unique solution of the martingale problem for $\Asv_0$. 
Let now
\begin{equation}
\begin{aligned}\label{findimsimplexinterior}\Delta^{\circ}_H:=\bigg\{w\in
\Delta_H:w_h>0,h=1,\ldots ,H\bigg\}.\end{aligned}
\end{equation} 
{Since completeness and separability of the state space will be used repeatedly in what follows, we record here a convenient metric under which $\Delta_H^\circ$ becomes a complete, separable metric space:}
\begin{equation}\label{Delta0metric}
{d(w,\tilde{w}):=\sqrt{\sum_{h=1}^H\bigg|\frac{1}{w_h}-\frac{1}{\tilde{w}_h}\bigg|^2}.}
\end{equation}
{Indeed, if $(w^n)$ is Cauchy for $d$, then for each $h=1,\ldots,H$ the sequence $(1/w_h^n)$ is Cauchy in $\mathbb R$, hence converges to some finite limit $a_h\ge 1$. Therefore $w_h^n\to a_h^{-1}>0$ for each $h$, and since $\sum_{h=1}^H w_h^n=1$ for all $n$, the limit belongs to $\Delta_H^\circ$. Separability is immediate since $w\mapsto (1/w_1,\ldots,1/w_H)$ is an isometric embedding into $\mathbb R^H$.}

The following lemma identifies a restriction in \eqref{eq:hatA0} that ensures that $\Delta_H\setminus \Delta^{\circ}_H$ acts as  an entrance boundary.

\begin{lemma}\label{lem:thetas}
Let $W^K$ be a WF diffusion with generator \eqref{eq:hatA0}. If $\theta_
h\ge 1$ for all $h=1,\ldots ,H$, then, almost surely, $W^K(t)\in\Delta^{
\circ}_H$ for all $t>0$. 
If, in addition, $W^K(0)\in\Delta^{\circ}_H$ almost surely, then
\begin{equation}\nonumber
\begin{aligned}
\PP(W^K(t)\in\Delta^{\circ}_H,\ \forall t\geq 0)=1.
\end{aligned}
\end{equation} 
\end{lemma} 
\begin{proof}
When $H=2$, it follows from \cite{KT81}, Example 8, p.~239, that if $
\theta_1,\theta_2\ge 1$, both 0 and 1 are entrance boundaries for $
W_1^K(t)\in [0,1]$. 
For $H>2$, the claim follows from reduction to the previous case. More specifically, for every $
h=1,\ldots ,H$, 
the single component $W^K_h$ is a WF diffusion on $[0,1]$ with mutation parameters $
\theta_h$ and $\theta_{(-h)}:=\sum_{k\ne h}\theta_k$. Cf.~\cite{D10}, Section 6.4. 
Hence, under the assumption on the parameters $\theta_h,\theta_{
(-h)}\ge 1$, 0 and 1 are entrance boundaries for $W^K_h$ as well. 
It follows that for no $h=1,\ldots ,H$, $W^K_h$ ever touches 0 or 1, 
hence $W^K(t)\in\Delta^{\circ}_H$ for all $t>0$ almost surely. The second claim is now obvious.
\end{proof}

We assume henceforth that
\begin{equation}
\begin{aligned}\label{eq:entrbdry}
\theta_{h}\geq 1, \quad \quad h=1,\ldots,H,
\end{aligned}
\end{equation}
and, unless otherwise stated, that $W^K(0)\in\Delta_H^\circ$ almost surely.
Then we may regard $W^K$ as a diffusion with state space $\Delta_H^\circ$, with generator
\eqref{eq:hatA0} acting on $\{f:\ f=\tilde f|_{\Delta_H^\circ},\ \tilde f\in C^2(\Delta_H)\}$.
Since the set of polynomials of $H$ variables is a core for the operator \eqref{eq:hatA0} (Theorem 8.2.8 of \cite{EK86}), $W^K$ is also the
unique solution of the martingale problem for $\Asv_0$ with domain
\begin{equation}
\begin{aligned}
\label{eq:DA0}\D(\Asv_0):=\{w_1^{p_1}w_2^{p_2}\cdots w_H^{p_H},\;
p_1,\ldots ,p_H\in\Z,\,p_1,\ldots ,p_H\geq 0\}.
\end{aligned}
\end{equation}

Let now $\Xt_h^K$, $h=1,\ldots ,H$, be mutually independent WF 
diffusions, independent of $W^K$, each taking values in $\Delta_K$. 
$\Xt_h^K$ has generator 
\begin{equation}
\begin{aligned}
\label{eq:hatAh}
\AAA_{h}^Kf(x):=&\,\frac{1}{2}\sum_{i,j=1}^Kx_i(\delta_{i,j}-x_j)\frac{
\partial^2f}{\partial x_i\partial x_j}+\frac{\theta_h}2\sum_{i=1}^
K\bigg(\frac 1K-x_i\bigg)\frac{\partial f}{\partial x_i},\\
\D(\AAA_{h}^K):=&\,C^2(\Delta_K).
\end{aligned}
\end{equation} 
 Here mutation is symmetric (for simplicity of notation we are allowing self-mutations) and 
the reversible measure is a $\mathrm{Dir}_K(\theta_h/K,\ldots ,\theta_
h/K)$ distribution.

For $\Delta^{\circ}_{H}$ as in \eqref{findimsimplexinterior}, set now
\begin{equation}
\begin{aligned}\label{eq:EK}
E_K:
=\Delta^{\circ}_{H}\times\Delta_K^{H},
\end{aligned}
\end{equation} 
where $\Delta_K^H$ denotes the $H$-fold cartesian product of $\Delta_
K$.
For $w\in \Delta^{\circ}_{H}$, define
\begin{equation}\label{eq:beta}
\begin{aligned}
\beta_0(w):=1,\quad\quad \beta_h(w):=\frac 1w_{h},\quad h=1,\ldots,H,
\end{aligned}
\end{equation} 
and define the random time-changed processes
\begin{equation}
\begin{aligned}
\label{eq:Xh}{\mathcal{X}}_h^K(t):=\Xt_h^K\bigg(\int_0^t\beta_h(W^K(s))ds\bigg
),\;\;h=1,\ldots ,H.\end{aligned}
\end{equation} 
Since $W^K$ takes values in $\Delta^{\circ}_{H}$, ${\mathcal{X}}_h^K$ is well defined for all times.

Consider the operator $\A^K$ of the form 
\begin{equation}
\begin{aligned}\label{eq:AK}
\A^Kf(w,x):=\Amv_0f(w,x)+\sum_{h=1}^H\beta_h(w)\Amv_
h^Kf(w,x),
\end{aligned}
\end{equation} 
where $w=(w_{1},\ldots,w_{H})\in \Delta_{H}$, $x=(x_{1},\ldots,x_{H})$, $x_{h}\in \Delta_{K}$, $h=1,\ldots,H$.
Here, with a slight abuse of notation, we still denote by 
$\Amv_0$ the operator that acts on $f$ as a function of $w$ like $\Asv_0$ in \eqref{eq:hatA0}, and, for each $h=1,\ldots ,H$, 
by $\AAA_{h}^K$ the operator that acts on $f$ as a function of $x_h$ like  $\AAA_{h}^K$ in \eqref{eq:hatAh}.

We will consider two domains for $\A^{K}$. The first one is defined as
\begin{equation}
\begin{aligned}
\label{eq:DK0}
\D_0(\A^K):=\bigg\{&f=\prod_{h=0}^Hf_h:\,f_0\in\D(\AAA_0),\,f_h\in\D(\AAA_{h}^K),\\
&\,\sup_{(w,x)\in E_K}f_0(w)\beta_h(w)|\AAA_{h}^K f_h(x_h)|<\infty,\,h=1,\ldots ,H\bigg\},
\end{aligned}
\end{equation} 
for $\D(\Asv_0)$ as in \eqref{eq:DA0} and $\D(\AAA_{h}^K)$ as in \eqref{eq:hatAh}. 
The second domain is defined as 
\begin{equation}
\begin{aligned}
\label{eq:DK}\D(\A^K):=\bigg\{&\,f:\,f=\tilde {f}\big\vert_{E_K},\,\tilde{
f}\in C^2\big(\overline {E_K}\big),\\
&\,\sup_{(w,x)\in E_K}\beta_h(w_
h)|\AAA_{h}^K\tilde {f}(w,x)|<\infty ,\,h=1,\ldots ,H\bigg\},\end{aligned}
\end{equation} 
with $E_K$ as in \eqref{eq:EK}.
Note that 
\begin{equation}
\begin{aligned}\label{eq:DK0sub}
\D _0(\A^{K})\subseteq \D (\A^{K}).
\end{aligned}
\end{equation} 
The following result shows that
$(W^K,\mathcal{X}_1^K,\ldots ,\mathcal{X}_H^K)$, with $\mathcal{X}_
h^K$ as in \eqref{eq:Xh}, solves the martingale problem for $\A^K$. 

\begin{theorem}\label{th:AKmpex}
Let $W^K$ be a WF diffusion with generator given by 
$(\ref{eq:hatA0})$ and $(\ref{eq:DA0})$. For 
$h=1,\ldots ,H$ and $\Xt^K_h$  mutually independent WF diffusions with generator \eqref{eq:hatAh}, independent of $W$, let $\mathcal{
X}_h^K$ be as in \eqref{eq:beta}-\eqref{eq:Xh}. 
Then $(W^K,\mathcal{X}_1^K,\ldots ,\mathcal{X}_H^K)$ is a solution of the 
martingale problem for $(\A^{K},\D (\A^{K}))$ as in \eqref{eq:AK}-\eqref{eq:DK}, and hence 
of the martingale problem for $(\A^{K},\D _0(\A^{K}))$ with $\D _0(\A^{K})$ as in \eqref{eq:DK0}. 
\end{theorem}

\begin{proof}
In order to fit into the framework of Section 2, Chapter 6 of \cite{EK86},  which is formulated for a countable family of processes, for $h>H $ we add independent WF diffusions $\Xt _h^K$, with 
generator $(\ref{eq:hatAh})$ and parameter $\theta_h=1$, independent of $(W^K,\Xt_1^K,\ldots ,
X_H^K)$. 
Given  $0<\delta\leq 1$, define the functions
\begin{equation}
\begin{aligned}\label{beta delta functions}
\beta_0^{\delta}(w):=
1,\quad\quad 
\beta_h^{\delta}(w_{h}):=\frac{1}{w_{h}\vee\delta},\ h=1,\ldots,H, \quad \quad 
\beta_h^{\delta}(w_{h}):=0,\ h>H ,
\end{aligned}
\end{equation} 
and let
\[\mathcal{X}^{K,\delta}_0(t):=W^K(t),\quad\mathcal{X}_h^{K,\delta}
(t):=\Xt_h^K\bigg(\int_0^t\beta_h^{\delta}(\Xtau_0^{K,\delta})(s)
)d {s}\bigg),\;\;h\geq 1.\]
Define  also
\begin{equation}
\begin{aligned}
\label{AK,delta}\A^{K,\delta}f(w,x):=\Amv_0f(w,x)+\sum_{h=1}^H\beta_
h^{\delta}(w_h)\AAA_{h}^Kf(w,x),\end{aligned}
\end{equation} 
with domain 
\[\D(\A^{K,\delta}):=\bigg\{f:\,f=\tilde {f}\big\vert_{E_K},\,\tilde {
f}\in C^2\big(\overline {E_K}\big)\bigg\}.\]
The vector $\Xtau^{K,\delta}:=(\Xtau^{K,\delta}_0,\mathcal{X}_1^{
K,\delta},\ldots ,\mathcal{X}_H^{K,\delta},\mathcal{X}_{H+1}^{K,\delta}
,\ldots )$ is a solution of the system (6.2.1) in \cite{EK86}, 
with the corresponding notation $Y_0:=W^K$ and 
$Y_h:=\Xt_h^K$ for $h\geq 1$, and where we have applied the trivial 
time-change determined by $\beta_0^{\delta}$ to $W^K$ (the subsequent citations in this proof will be from the same source). It is immediate that 
$\Xtau^{K,\delta}$ is the pathwise unique solution. 
Therefore, by Theorem 6.2.2, 
and following definitions), we have that
$\tau^{\delta}(t):=(\tau_0^{\delta}(t),\,\tau_1^{\delta}(t),\,\tau_
2^{\delta}(t),\,\ldots )$, defined by
\[\tau_h^{\delta}(t):=\int_0^t\beta_h^{\delta}({\mathcal{X}}^{K,\delta}_0(s))
d {s},\]
is an $\big\{{\cal F}_u\big\}_{u\in [0,\infty )^{\infty}}$-stopping time for every $
t\geq 0$ (cf.~also eqn.~(6.2.7)). This trivially implies 
that $\Xtau^{K,\delta}$ is a nonanticipating solution (cf.~(6.2.15)). Then, by Theorem 6.2.8 (a),
\[
\nonumber
\prod_{j=0}^Hf_j(\mathcal{X}_j^{K,\delta}(t))-\int_0^t\sum_{h=0}^
H\beta_h^{\delta}(\mathcal{X}_0^{K,\delta}(s))\prod_{j\in \{0,\ldots 
,H\},j\neq h}\;f_j(\mathcal{X}_j^{K,\delta}(s))\,\AAA_{h}^Kf_h(\mathcal{
X}_h^{K,\delta}(s))ds\]
is a martingale for every $f_0$ such that $f_0=\tilde{f}_0 \vert_{\Delta_H^{\circ}}$, $\tilde{f}_0\in C^2(\Delta_H)$ and $f_h\in 
\D (\Asv^K_h)$, $h=1,\ldots,H$. Therefore 
\begin{equation}
\begin{aligned}
f\big(\mathcal{X}_0^{K,\delta}(t),\ldots ,\mathcal{X}_H^{K,\delta}
(t)\big)-\int_0^t\A^{K,\delta}f\big(\mathcal{X}_0^{K,\delta}(s),\ldots 
,\mathcal{X}_H^{K,\delta}(s)\big)d {s}\label{eq:deltamg}\end{aligned}
\end{equation} 
is a martingale for every $f$ that is a linear combination of products as above.
Since polynomials are dense in 
$C^2\big(\overline {E_K}\big)$ (in 
the norm $\|f\|:=\sum_{|\lambda |\leq 2}\sup_{\overline {E_K}}|D^{
\lambda}f(w,x_1,x_2,\ldots .,x_H)|$, for $\lambda$ a 
multi-index), it is also a martingale for every 
$f\in\D(\Amv^{K,\delta})$.

Let now $\{\delta_n\}_{n\ge 1}\subset\R_{+}$ be a sequence  decreasing to 
$0$. As $n\rightarrow\infty$, for each
$h=1,\ldots ,H$ the sequence $\beta_h^{\delta_n}(W^K(\cdot ))$ converges uniformly over compact time intervals to $
\beta_h(W^K(\cdot ))$, almost surely. 
Then also $(W^K,\mathcal{X}_1^{K,\delta_n},\ldots ,\mathcal{X}_H^{K
,\delta_n})$ 
converges uniformly over compact time intervals to 
$(W^K,\mathcal{X}_1^K,\ldots ,\mathcal{X}_H^K)$,  
almost surely,
and for every $f\in\D(\A^K)$, we have, as $n\rightarrow\infty$, the convergence
\[\A^{K,\delta_n}f(W^K(\cdot ),\mathcal{X}_1^{K,\delta_n}(\cdot ),\ldots 
,\mathcal{X}_H^{K,\delta_n}(\cdot ))\rightarrow\A^Kf(W^K(\cdot ),\mathcal{
X}_1^K(\cdot ),\ldots ,\mathcal{X}_H^K(\cdot )),\]
in the bounded and pointwise sense, almost surely. 
Therefore, almost surely, for each $t\geq 0$, the martingale 
\[f(W^K(t),\mathcal{X}_1^{K,\delta_n}(t),\ldots ,\mathcal{X}_H^{K,\delta_
n}(t))-\int_0^t\A^{K,\delta_n}f(W^K(s),\mathcal{X}_1^{K,\delta_n}(s
),\ldots ,\mathcal{X}_H^{K,\delta_n}(s))d {s}\]
converges boundedly and pointwise to 
\[f(W^K(t),{\mathcal{X}}_1^K(t),\ldots,{\mathcal{X}}_H^K(t))-\int_0^t\A^{K}f(W^K(s),{\mathcal{X}}_1^K(s),\ldots,{\mathcal{X}}_H^K(
s))d {s}\]
and the limit is a martingale.
\end{proof}

The following result adds uniqueness for the solution in Theorem \ref{th:AKmpex}. 
Denote by $\PPP (G)$ the space of probability measures on a generic metric space $
G$ and by $D_{G} [0,\infty)$ the set of right continuous functions with left hand limits from $[0,\infty)$ to $G$. 
We will also use the notation for product measures $\bigtimes_{n=1}^m\,\mu_n:=\mu_1\times\cdots\times\mu_m$. 
Finally, for a set of real valued functions $\cal G$, $\rm{span}\big(\cal G\big)$ denotes the linear space of finite linear combinations of functions of $\cal G$.

\begin{theorem}\label{th:AKunico}
For each $\mu\in {\cal P}(E_K)$ of the form $\mu =\bigtimes_{h=0}^
H\,\mu_h$,
there is at most one solution of the 
martingale problem for $(\A^{K},\D_0(\A^{K}))$ with initial distribution 
$\mu$, and hence 
of the martingale problem for $(\A^{K},\D (\A^{K}))$ with initial 
distribution $\mu$. 
\end{theorem}
\begin{proof}
Let $(W^K,\mathcal{X}_1^K,\ldots,\mathcal{X}_H^K)$ be a solution of the martingale
problem for $(\A^{K},\D_0(\A^{K}))$. Then $W^K$ is a solution of the martingale problem for $\Amv_0$ and, by our standing assumption \eqref{eq:entrbdry} and Lemma \ref{lem:thetas}, satisfies the compact containment condition. Therefore $(W^K,\mathcal{X}_1^K,\ldots,\mathcal{X}_H^K)$ also satisfies the compact containment condition and, by Theorem 4.3.6 of \cite{EK86}, has a version with paths in $D_{E_K} [0,\infty)$, which we still denote in the same way. Let $(W^K,\mathcal{X}_1^K,\ldots,\mathcal{X}_H^K)$ have initial distribution $\mu$.
To place the argument within the multiparameter random time change framework of \cite{EK86}, which is formulated for a countable family of processes, 
we embed this $(H\!+\!1)$-dimensional process into an infinite dimensional process by adding
frozen auxiliary coordinates.
More precisely, for each $h>H$, fix an arbitrary $\mu_h\in\PPP(\Delta_K)$ and define
$\mathcal{X}_h^K(t):=\mathcal{X}_h^K(0)$ for all $t\ge 0$, where $\mathcal{X}_h^K(0)\sim\mu_h$
and $\{\mathcal{X}_h^K(0)\}_{h>H}$ are mutually independent and independent of
$(W^K,\mathcal{X}_1^K,\ldots,\mathcal{X}_H^K)$.
Set also $\mathcal{X}_0^K:=W^K$.
Then 
$\mathcal{X}^K:=(\mathcal{X}^K_0,\mathcal{X}_1^K,\ldots ,\mathcal{
X}_H^K,\Xtau^K_{H+1},\ldots )$ 
is a solution of the martingale problem for the operator  
\begin{equation}
\begin{aligned}
\label{calAK}
\A_{\infty}^Kf(w,x):=\Amv_0f(w,x)+\sum_{h\geq 1}\beta_
h(w)\AAA_{h}^Kf(w,x),
\end{aligned}
\end{equation} 
with $\Amv_0$, 
$\AAA_{h}^K$ as in \eqref{eq:AK} and $\beta_h$ as in $(\ref{eq:beta}
)$ for 
$h\leq H$, $\AAA_{h}^K$ as in 
$(\ref{eq:hatAh})$ with parameter $\theta_h=1$ and $\beta_h\equiv 0$ for $h>H$,  
with domain 
\begin{equation}
\begin{aligned}
\nonumber
\D(\A_{\infty}^K):=\bigg\{&f=f_0\prod_{h\in I}f_h:\;f_0\in\D(\Asv_
0),\,f_h\in\D(\AAA_{h}^K),&\,\\
&\,\sup_{(w,x)\in E_K}f_0(w)\beta_h(w_
h)|\AAA_{h}^K f_h(x_h)|<\infty,\,h\in I,\,I\subseteq \{1,2,\ldots \},\,I\mbox{\rm \ finite}\bigg
\}.
\end{aligned}
\end{equation} 
The thesis would follow immediately if we could apply Theorem 6.2.8 of \cite{EK86}, but this is not possible because the functions $\beta_h$ are unbounded. However, we show next that we can apply Proposition 6.2.10 of  \cite{EK86}, which allows to reduce to bounded random time-change rates. Define 
\[\alpha (w):=\bigg(\prod_{h=1}^Hw_h\bigg)^{-1},\qquad w\in\Delta^{\circ}_H
,\]
and let $\eta$ be the stochastic process defined pathwise by 
\[\int_0^{\eta (t)}\alpha (\Xtau^K_0(s))ds=t,\qquad t\geq 0.\]
By our standing assumption \eqref{eq:entrbdry} and Lemma \ref{lem:thetas},  
\[\int_0^T\alpha (\Xtau^K_0(s))ds<\infty ,\qquad\forall T\geq 0,\]
so that $\lim_{t\rightarrow\infty}\eta (t)=\infty$ almost surely and the assumptions of 
Proposition 6.2.10 of \cite{EK86} are satisfied. Let 
\[\mathcal{X}^{K,\alpha}(t):=\Xtau^K(\eta (t)),\quad t\geq 0.\]
Then 
\[\eta (t)=\int_0^t\frac{1}{\alpha ({\mathcal{X}}^{K,\alpha}_0(s))}d s\]
and 
\begin{equation}
\begin{aligned}
f(\mathcal{X}^{K,\alpha}(t))-\int_0^t\frac{1}{\alpha (\mathcal{X}^{
K,\alpha}_0(s))}\A_{\infty}^Kf\,(\mathcal{X}^{K,\alpha}(s))d {s}\label{eq:alphamp}\end{aligned}
\end{equation} 
is a martingale for every $f\in\D(\A_{\infty}^K)$. Let 
\[\A_{\infty,b}^Kf(w,x):=\alpha (w)^{-1}\Amv_0f(w,x)+\sum_{h\geq 1}\alpha (w)^{-1}\beta_h(w)\AAA_{h}^Kf(w,x)
,\]
$ $with domain 
\begin{equation}
\begin{aligned}
\nonumber
\label{eq:basicD}\D\big(\A_{\infty,b}^K&\big):=\bigg\{f=f_0\prod_{
h\in I}f_h:\;\,f_0\in\D(\Asv_0),\ f_h\in\D(\AAA_{h}^K),\,h\in I,\ I
\subseteq \{1,2,\ldots \},\,I\mbox{\rm \ finite}\bigg\}.
\end{aligned}
\end{equation} 
 
Every $f\in\D(\A_{\infty,b}^K)$ can be 
approximated pointwise and boundedly by functions 
$\{f^n\}\subseteq\rm{span}\big(\D(\A_{\infty}^K)\big)$ in such a way that, defining $\A_{\infty}^Kf^n$ by linearity, 
$\{\alpha^{-1}\A_{\infty}^Kf^n\}_{n\ge 1}$ converges 
pointwise and boundedly to $\A_{\infty,b}^K f$. (This follows from the fact that, for each $\tilde{f}_0\in C^2(\Delta_H)$, for each $n\in \N$, the function 
\begin{equation}
\nonumber
\tilde{f}_0^n(w):=\frac {\tilde{f}_0(w)\rho (nw_1\ldots w_H-1)}{w_1\ldots w_H},
\end{equation}
with $\rho:\R\rightarrow [0,1]$ a smooth, nondecreasing function such that $\rho (u)=0$ for $u\leq 0$, $\rho (u)=1$ for $u\geq 1$, can be approximated by polynomials in the norm $\|f\|:=\sum_{|\lambda |\leq 2}\sup_{\Delta_H}|D^{
\lambda}f(w)|$, for $\lambda$ a 
multi-index.) 

Then $\mathcal{X}^{K,\alpha}$ is a solution of the martingale problem for $\A_{\infty,b}^K$. 
Since $\alpha^{-1}$ and $\alpha^{-1}\beta_
h$, for $h\geq 1$, are bounded functions, we can apply Theorem 6.2.8 (b) of \cite{EK86} 
(cfr.~also page 314) to obtain that there is a 
version $\tilde{\Xtau}^{K,\alpha}$ of $\mathcal{X}^{K,\alpha}$ that is a nonanticipating solution of 
\[\tilde{\Xtau}^{K,\alpha}_0(t)={\tilde W}^K\bigg(\int_0^t\frac{1}{\alpha 
(\tilde{\Xtau}_0^{K,\alpha}(s))}ds\bigg),\quad\tilde{\Xtau}^{K,\alpha}_
h(t)=\tilde{\Xt}^K_h\bigg(\int_0^t\frac{\beta_h(\tilde{\Xtau}_0^{
K,\alpha}(s))}{\alpha (\tilde{\Xtau}_0^{K,\alpha}(s))}ds\bigg),\;\;
h\geq 1,\]
where ${\tilde W}^K$ and $\big\{\tilde{\Xt}_h^K\big\}_{h\geq 1}$ are independent processes with 
generators $\Asv_0$ and $\big\{\AAA_{h}^K\big\}_{h\geq 1}$, and initial distributions $
\mu_0$ 
and $\{\mu_h\}_{h\geq 1}$.  
Applying Proposition 6.2.10 of \cite{EK86} to $\tilde{\Xtau}^{K,\alpha}$, we obtain that 
there is a version $\tilde{\Xtau}^K$ of $\Xtau^K$ that is nonanticipating and satisfies 
\begin{equation}\label{eq:nonantunq}
\tilde{\Xtau}^K_0(t)=\tilde {W}^K(t),\quad\tilde{\Xtau}^K_h(t)=
\tilde{\Xt}_h^K\bigg(\int_0^t\beta_h(\tilde{\Xtau}_0^K(s))ds\bigg
),\;\;h\geq 1.
\end{equation}
If we view \eqref{eq:nonantunq} as an equation in the unknown $\big\{\tilde{\Xtau}_h^K\big\}_{h\geq 0}$, given 
${\tilde{W}}^K$ and $\big\{\tilde{\Xt}_h^K\big\}_{h\geq 1}$, then it has a pathwise unique solution, because 
$\tilde{\Xtau}^K_0=\tilde{W}^K$ is given and, for each $h\geq 1$, $\tilde{\Xtau}^K_h$ 
is uniquely determined by $\tilde{\Xt}_h^K$, which is also given, and $\tilde{W}^K$. In particular, the law of $\tilde{\Xtau}^{K}$, i.e. the law of $\Xtau^{K}$, is uniquely determined. 
\end{proof}

\begin{corollary}\label{th:Kmixt}
For every $\mu\in {\cal P}(E_K)$, the martingale problem for $\A^{K}$,  $\A^{K}$ as in \eqref{eq:AK}-\eqref{eq:DK}, with initial distribution $\mu$, has one and only one solution. The law of the solution can be represented as 
\begin{equation}\label{eq:Kmixt}
\int_{E_K}P^K_{(w,x)}\, d\mu ,
\end{equation}
where $P^K_{(w,x)}$ is the law of the solution of the martingale problem for $\A^{K}$ with initial distribution $\delta_{(w,x)}$
\end{corollary}

\begin{proof}
As in the proof of Theorem \ref{th:AKunico}, by Theorem 4.3.6 of \cite{EK86} we can reduce to solutions with paths in $D_{E_K} [0,\infty)$. 
In order to see that the martingale problem for $\A^{K}$ with initial distribution $\mu$ has at least one solution for every $\mu\in {\cal P}(E_K)$, 
one can observe that $P^K_{(w,x)}$ is continuous in the weak convergence topology (this follows from Theorems \ref{th:AKmpex} and \ref{th:AKunico} and the properties of WF diffusions), hence measurable, and that any probability measure on $D_{E_K} [0,\infty)$ of the form \eqref{eq:Kmixt} is a solution. 
Uniqueness and \eqref{eq:Kmixt} follow from Theorem 4.5.19 of \cite{EK86}. 
In order to verify the assumptions of Theorem 4.5.19, one can use Theorem 4.5.11 (b) and Lemma 4.5.13 of \cite{EK86}: with the notation of  \cite{EK86}, for each compact $K$ and $\eta>0$, setting 
\begin{equation}
\begin{aligned}
\nonumber
K_R^{\circ}:=&\,\{w\in\Delta_H^{\circ}:d(w,w^*)\leq R \},\quad w^*:=(H^{-1},\ldots,H^{-1}), \\
K_n:=&\,K_{\sqrt{n}}^{\circ}\times \Delta_K^H,\quad n\geq n_0,\, n_0\;\textrm{ sufficiently  large that}\;K\subseteq K_{n_0},
\end{aligned}
\end{equation}
where $d$ is given by \eqref{Delta0metric}, and 
\begin{equation}
\nonumber
f_n(w,x)=f_n(w),\;\forall (w,x)\in E_K,
\end{equation}
with $f_n \in \C^2_b(\Delta_H^{\circ})$ such that 
\begin{equation}
\nonumber
\begin{array}{ll}f_n(w):=1+\log(1+(\sqrt{n}+\eta)^2)-\log(1+d(w,w^*)^2),&{\rm for} \;w\in K_{\sqrt{n}+\eta}^{\circ},\\
f_n(w):=0,&{\rm  for}\, w\in \Delta_H^{\circ}\setminus K_{\sqrt{n}+\eta+1}^{\circ},\\
 0\leq f_n(w)\leq 1, &{\rm  for}\, w\in \Delta_H^{\circ}\setminus K_{\sqrt{n}+\eta}^{\circ},\end{array}
\end{equation}
$(4.5.34)$, $(4.5.35)$, and $(4.5.36)$ of \cite{EK86} are satisfied with $g_n:=A_0 f_n$. 
\end{proof}


\subsection{Blockwise skew-product representation}\label{skprodform}

Let $\Delta_{HK}$ be defined as in \eqref{findimsimplex}. Recall that we view all vectors as row 
vectors. For $z\in\Delta_{HK}$, let 
\begin{equation}\label{eq:zHK}
z=(z_1,\ldots ,z_H),\quad z_h=(z_{h,1},\ldots ,z_{h,K}),\quad 
z_{h,i}:=z_{(h-1)K+i},
\end{equation} 
and
\begin{equation}\label{eq:DeltaK2}
\Delta^{H,\circ}_{HK}:=\bigg\{(z_1,\ldots ,z_H)\in\Delta_{HK}:\,
\sum_{i=1}^Kz_{h,i}>0,\ h=1,\ldots ,H\bigg\},
\end{equation} 
and define the map $S:E_K\rightarrow\Delta^{H,{\circ}}_{H
K}$ ($E_K$ as in \eqref{eq:EK}) by  
\begin{equation}
\begin{aligned}S(w,x):=wx=(w_{1}x_{1},\ldots,w_{H}x_{H}),\quad \quad 
w_{h}x_{h}=(w_{h}x_{h,1},\ldots,w_{h}x_{h,K}).
\label{eq:Smap}
\end{aligned}
\end{equation} 

The following lemma is immediate.

\begin{lemma}\label{th:K-Sinv}
$S$ is bijective and $S^{-1}$ is given by 
\begin{equation}
\begin{aligned}
S^{-1}(z)=\Bigg(\bigg(\sum_{i=1}^Kz_i,\ldots ,\sum_{i=1}^Kz_{H,i}\bigg
),\frac{z_1}{\sum_{i=1}^Kz_{1,i}},\ldots ,\frac{z_H}{\sum_{i=1}^Kz_{
H,i}}\Bigg).\label{eq:S-1map}\end{aligned}
\end{equation} 
$S$ and $S^{-1}$ are continuous. 
\end{lemma}

Let $\B^K$ be the generator of an $HK$-type WF diffusion. 
In view of \eqref{eq:zHK},  we can write $B^K$  as 
\begin{equation}
\begin{aligned}\label{eq:BK}
\B^Kf(z_{1,1},\ldots ,z_{H,K})
=&\,\frac{1}{2}\sum_{h,k=1}^H\sum_{i,
j=1}^Kz_{h,i}(\delta_{h,k}\delta_{i,j}-z_{k,j})\frac{\partial^2f}{
\partial z_{h,i}\partial z_{kj}}\\
&\,+\frac{1}{2}\sum_{h=1}^H\sum_{i=1}^K
\bigg(\frac{\theta_h}{K}-\tn z_{h,i}\bigg)\frac{\partial f}{\partial 
z_{h,i}},\\
\D(\B^K)=&\,C^2(\Delta_{HK}).
\end{aligned}
\end{equation} 
As is well known, 
the martingale problem for $\B^K$ is well posed (see, for 
instance, Theorems 8.2.8 and 4.4.1 in \cite{EK86}). 
Next, we formalize the connection between \eqref{eq:AK} and \eqref{eq:BK}. We first show a preliminary result, 
used in the subsequent theorem. 

\begin{lemma}\label{prop:generator mapping}
Let $\A^{K}$ be as in \eqref{eq:AK} and $\B^K$ be as in \eqref{eq:BK}. Then, for $f\in \D (\B^K)$ and $S$ as in \eqref{eq:Smap}, $f\circ S \in \D (\A^K)$ as in \eqref{eq:DK} and 
\begin{equation}
\begin{aligned}\label{AK=BK}
\A^{K}(f\circ S)=(\B^Kf)\circ S.
\end{aligned}
\end{equation} 
More generally, if $f$ is a function defined on $\R^{HK}$ that has second derivatives at each $z\in \Delta_{HK}^{H,\circ}$, then, viewing $f\circ S$ as a function on $\R^H\times \R^{HK}$, $f\circ S$ has second derivatives at each $(w,x)\in E_K$ and \eqref{AK=BK} holds at each $(w,x)\in E_K$. 
 \end{lemma}
\begin{proof}
The proof consists in a long but straightforward computation, therefore it is omitted.
\end{proof}

We are now ready to state the main result of this 
section. By leveraging on the previous Lemma, 
the following Theorem shows that the solution of the martingale problem for 
$\B^K$ can be represented as a multiparameter random time 
change of $H$ independent $K$-type WF diffusions with generators \eqref{eq:hatAh}, 
each weighted by the corresponding component of 
an $H$-type WF diffusion with generator $(\ref{eq:hatA0})$, 
which also determines the random time rescaling. 
This amounts to a self-similarity property for WF diffusions, 
analogous to that of Proposition \ref{pro:Dirselfsimilarity}, and provides a \emph{blockwise skew-product representation} of the WF diffusion $Z^K$.

\begin{theorem}\label{th:Krepr}
Let $Z^K$ be the solution of the martingale 
problem for $\B^K$ as in $(\ref{eq:BK})$, with initial state $Z^
K(0)=(Z^K_1(0),\ldots ,Z^K_H(0))\in\Delta^{H,\circ}_{HK}$ almost surely. Let 
\begin{equation}
\begin{aligned}
\label{definitialcondition}W^K_h(0):=\sum_{i=1}^KZ_{h,i}^K(0),\quad\quad 
W^K(0)=(W^K_1(0),\ldots ,W^K_H(0)),\end{aligned}
\end{equation} 
and assume that 
\begin{equation}
\begin{aligned}
\label{requirementsoninitialcondition}W^K(0)\,,\,\frac{Z_1^K(0)}{W^K_
1(0)}\,,\ldots ,\,\frac{Z_H^K(0)}{W^K_H(0)}\quad\mbox{\rm are mutually independent}
.\end{aligned}
\end{equation} 
Then 
\begin{equation}\label{eq:Krepr}
\begin{aligned}
Z^K
\stackrel 
d{=}\big(W^K_1(\cdot )\,\mathcal{X}_1^K(\cdot ),\ldots ,W^K_H(\cdot 
)\,\mathcal{X}_H^K(\cdot )\big),\end{aligned}
\end{equation} 
where, for $h=1,\ldots ,H$, $\mathcal{X}_h^K(t)=\Xt_h^K(\int_0^t\beta_
h(W^K(s))ds)$ is as in \eqref{eq:Xh}, and $W^K,X_1^K,\ldots ,\linebreak\Xt_H^K$ are 
mutually independent with generators 
\eqref{eq:hatA0}-\eqref{eq:DA0} and \eqref{eq:hatAh} and initial conditions \break 
$W^K(0),\Xt_1^K(0):=Z^K_1(0)/W^K_1(0),\ldots ,\Xt_H^K(0):=Z^K_H(0)/W^K_
H(0)$. 

More generally, if $Z^K$ is the solution of the martingale 
problem for $\B^K$ as in $(\ref{eq:BK})$, with initial distribution $\nu\in {\cal P}\big(\Delta_{HK}\big)$ such that 
$\nu\big(\Delta^{H,\circ}_{HK}\big)=1$, then \eqref{eq:Krepr} holds with 
$\big(W^K,\mathcal{X}^K\big)$ the solution of the martingale problem for 
$\A^{K}$ as in \eqref{eq:AK}-\eqref{eq:DK}, with initial distribution $\nu\circ S$ defined as: $(\nu\circ S)(C):=\nu (S(C))$, $C$ a Borel set of $E_K$ (note that $S(C)$ is a Borel set of $\Delta_{HK}$ by Lemma \ref{th:K-Sinv}).
\end{theorem}

\begin{proof}
By Lemma \ref{prop:generator mapping}, for  $f\in\D(\B^K)$, 
we have $f\circ S\in\D(\A^K)$ and $\A^K(f\circ S)=\B^Kf\circ S$. 
Therefore 
\[f\big(W^K_1(t)\,\mathcal{X}_1^K(t)\,,\ldots ,\,W^K_H(t)\,\mathcal{X}_
H^K(t)\big)-\int_0^t\B^Kf\big(W^K_1(s)\,\mathcal{X}_1^K(s)\,,\ldots 
,\,W^K_H(s)\,\mathcal{X}_H^K(s)\big)d {s}\]
is a martingale, hence the process 
\[\big(W^K_1\,\mathcal{X}_1^K,\ldots ,W^K_H\,\mathcal{
X}_H^K\big)\]
is a solution of the 
martingale problem for $\B^K$ with initial condition 
$(Z^K_1(0),\ldots,Z^K_H(0))$. By virtue of the uniqueness of the solution to the martingale 
problem for $\B^K$, we conclude that $\big(W^K_1\,\mathcal{X}_1^K,\ldots ,W^K_H\,\mathcal{
X}_H^K\big)$ equals $(Z^K_1,\ldots ,Z^K_H)$ in distribution. 
\end{proof}

In the next section we will show that such a representation carries over to the infinite-dimensional setting, 
when the number of types for each mark is let go to infinity.



\section{Multiple Poisson--Dirichlet diffusions}\label{sec:convergence}

In this section we derive the infinite-dimensional limit of the WF diffusions $Z^K$ of Section \ref{sec:hierarchicalWF} upon decreasingly ranking the components within each mark.

{We first prove convergence of the ranked processes by exploiting their blockwise skew-product representation (Section~\ref{subsec:ranked-conv}). The limit is characterized in skew-product coordinates as the unique solution of a martingale problem (Section~\ref{subsec:gen-on-E}).}
In Section \ref{subsec:gen-on-simplex}, we rewrite the operator of the martingale problem in generalized Kingman simplex coordinates, by pushing forward through the map $S$ in \eqref{eq:Smap}, and we show that it is the limit of the generators of the blockwise ranked WF diffusions. Next we give an explicit description of this operator and of its domain.

Finally,  in Section~\ref{subsec:domain-subtlety} we collect some technical remarks on our convergence proof and highlight why, in this setting, the limit operator cannot be viewed as an operator on $\C(\simplexH)$ ($\simplexH$ as in \eqref{extendedKingmansimplex}), as was the case, for $H=1$, in \cite{EK81,P09, Cea17}.


\subsection{Convergence of the blockwise ranked WF diffusions}\label{subsec:ranked-conv}

In this section we consider the problem of obtaining the limit, as $K\rightarrow \infty$, for $Z^{K}$ (as in Theorem \ref{th:Krepr}) after ranking decreasingly
the frequencies within each mark.
For $\Delta^{\circ}_H$ as in \eqref{findimsimplexinterior}, set
\begin{equation}
\begin{aligned}
\label{eq:E}
E:=\Delta^{\circ}_H\times\simplex^H,\quad\quad E^{\circ}
:=\Delta^{\circ}_H\times\simplexone^H,
\end{aligned}
\end{equation}
where $\simplex^H$ and $\simplexone^H$ are the $H$-fold cartesian
products of $\simplex$ as in \eqref{simplex} and of
$\simplexone$ as in \eqref{simplexone}, respectively.

With a slight abuse of notation, we will still denote by 
$S$ the map from $E$ to $\simplexH$, 
where $\simplexH$ is the generalized Kingman simplex in \eqref{extendedKingmansimplex}, defined as 
\begin{equation}
\begin{aligned}
\label{mapS}S(w,x):=(w_1x_1,\ldots ,w_Hx_H),\qquad w_hx_h=(w_hx_{
h,1},w_hx_{h,2},\ldots ),\quad x_h\in\simplex.\end{aligned}
\end{equation} 
Let 
\begin{equation}
\begin{aligned}
\label{extendedKingmansimplexinterior}
\simplexHo:=\bigg\{z\in\simplexH
:\,\sum_{h=1}^H\sum_{i=1}^{\infty}z_{h,i}=1,\;0<\sum_{i=1}^{\infty}
z_{h,i},h=1,\ldots ,H\bigg\}.
\end{aligned}
\end{equation} 

\begin{lemma}\label{th:Sinv}
The map $S:E\rightarrow\simplexH$, is continuous. 
\begin{equation}\label{eq:S-1simplexHo}
S^{-1}\big(\simplexHo\big)=E^{\circ},
\end{equation}
and $S$ is bijective from $E^{\circ}$ (cf.~\eqref{eq:E}) to $\simplexHo$. 

The maps $z\rightarrow\sum_{i=1}^{\infty}z_{h,i}$, $h=1,\ldots ,H$, are continuous on 
$\simplexHo$ and the map $S^{-1}:\simplexHo\rightarrow E^{\circ}$ is continuous. 
\end{lemma}
\begin{proof}
The first three assertions are immediate. The map $z\rightarrow\sum_{i=1}^{\infty}z_{h,i}$ is not continuous on 
$\simplexH$, but the following argument shows that it is 
continuous on $\simplexHo$. Let $\{z^n\}$ be a sequence of 
points in $\simplexHo$ converging to a point $z\in\simplexHo$, 
and set 
\[b_h:=\sum_{i=1}^{\infty}z_{h,i},\quad \quad b_h^n:=\sum_{i=1}^{\infty}
z_{h,i}^n.\]
We have 
\[b_h\leq\liminf_{n\rightarrow\infty}b_h^n,\;\;h=1,\ldots ,H,\qquad 
1=\sum_{h=1}^Hb_h=\sum_{h=1}^Hb_h^n.\]
Then 
\[b_H=1-\sum_{h=1}^{H-1}b_h\geq 1-\liminf_{n\rightarrow\infty}\sum_{
h=1}^{H-1}b_h^n\geq\limsup_{n\rightarrow\infty}b_H^n,\]
so that 
\[b_H=\lim_{n\rightarrow\infty}b_H^n.\]
Now let us show that, for every $h$, $2\leq h\leq H$, 
\[b_l=\lim_{n\rightarrow\infty}b_l^n\;\;\forall l:\,h\leq l\leq H
\quad\implies\quad b_l=\lim_{n\rightarrow\infty}b_l^n\;\;\forall 
l:\,h-1\leq l\leq H.\]
In fact the above yields
\[\sum_{l=1}^{h-1}b_l=1-\sum_{l=h}^Hb_l=\lim_{n\rightarrow\infty}
\bigg(1-\sum_{l=h}^Hb_l^n\bigg)=\lim_{n\rightarrow\infty}\sum_{l=1}^{h-1}b_l^
n,\]
and hence 
\[b_{h-1}=\sum_{l=1}^{h-1}b_l-\sum_{l=1}^{h-2}b_l\geq\lim_{n\rightarrow
\infty}\sum_{l=1}^{h-1}b_l^n-\liminf_{n\rightarrow\infty}\sum_{l=
1}^{h-2}b_l^n\geq\limsup_{n\rightarrow\infty}b_{h-1}^n.\]
The last assertion is now obvious. 
\end{proof}

Let $\Delta_K^c$ be the corner of the hypercube $[0,1]^K$ defined as
\[\Delta_K^{c}:=\left\{z\in [0,1]^K:\ \sum_{j=1}^Kz_j\leq 1\right
\},\]
and consider the continuous map $\rho^{K}:\Delta_K^{c}\rightarrow \simplex$, defined as
\begin{equation}
\begin{aligned}
\label{eq:rho}\rho^{K}(x):=(x_{(1)},x_{(2)},\ldots, x_{(K)},0,0,\ldots )\end{aligned}
\end{equation} 
where $x_{(1)}\geq x_{(2)}\geq \ldots\geq x_{(K)}$ are the descending order 
statistics of the coordinates of $x_{h}\in\Delta_K^{c}$. 
Defining, for $z\in \Delta_{HK}$, 
\begin{equation}\label{rho_K vector operator}
\rho^{K}_{(H)}(z):=(\rho^K(z_1),\rho^K(z_2),...,\rho^K(z_H)),
\end{equation} 
and, for $(w,x)\in E_K$, 
\begin{equation}\label{rho_K vector operator E_K}
\rho^{K}_{(H)}(w,x):=
(w,\rho^{K}(x_{1}),\ldots,\rho^{K}(x_{H})),
\end{equation} 
we can write, recalling that $S$ denotes both the map in finite dimension \eqref{eq:Smap} and the analogous map in infinite dimension \eqref{mapS}, for $(w,x)\in E_K$, 
\begin{equation}
\begin{aligned}
\label{eq:Scomm}S\big(\rho^{K}_{(H)}(w,x)\big)=\rho^{K}_{(H)}\big(
S(w,x)\big),\end{aligned}
\end{equation} 
and, for $z\in \Delta^{H,\circ}_{HK}$, 
\begin{equation}\label{eq:S-1comm}
\rho^{K}_{(H)}(S^{-1}(z))=
S^{-1}(\rho^{K}_{(H)}(z)).
\end{equation}

Let now $Z^K:=(Z^K_1,\ldots ,Z^K_H)$ be the solution (unique in distribution) of the martingale 
problem for $\B^K$ as in \eqref{eq:BK}, with initial condition satisfying 
\eqref{definitialcondition}-\eqref{requirementsoninitialcondition}. By Theorem \ref{th:Krepr}, we have
\[Z^K\stackrel d{=}S\big(W^K,\mathcal{X}^K\big),\quad\quad\mathcal{
X}^K=(\mathcal{X}_1^K,\ldots ,\mathcal{X}_H^K)\]
with $\mathcal{X}_h^K$as in \eqref{eq:Xh}.
Then, by $(\ref{eq:Scomm})$, 
\begin{equation}
\begin{aligned}
\label{eq:S(W,rho(X))}
\rho^{K}_{(H)}(Z^K)
\overset{d}=S\big(\rho^{K}_{(H)}(W^K,\mathcal{X}^K)\big).\end{aligned}
\end{equation} 
\cite{EK81} showed that $\rho^{K}(\Xt _h^
K)$ converges in distribution, as $K\rightarrow\infty$, 
to the Poisson--Dirichlet diffusion with generator \eqref{IMNAoperator} with $\theta$ replaced by $\theta_{h}$ and domain the algebra generated by the functions 
\begin{equation}\varphi_1(x)\equiv 1,\quad\varphi_m(x):=\sum_{i=1}^{
\infty}x_i^m,\quad m\geq 2,\label{eq:varphi}\end{equation}
provided the initial distributions converge. The following theorem shows that this allows to identify the  limit in distribution for \eqref{eq:S(W,rho(X))}. 

\begin{theorem}\label{Z1,Z2aslimit}
Let $Z^K$ be the (unique in distribution) solution of the martingale 
problem for $\B^K$ as in \eqref{eq:BK} and let $W^K(0)$ be as in \eqref{definitialcondition}. Assume $W^K(0)$ and $Z^K(0)$ satisfy \eqref{requirementsoninitialcondition} and that, as $K\rightarrow\infty$, 
\begin{equation}
\begin{aligned}\label{eq:initconv}
\big(W^K(0),\rho^K_{(H)}\big(Z^K(0)\big)\big)\overset{d}\longrightarrow \big(W(0),Z(0)\big),
\end{aligned}
\end{equation} 
 with $W(0)\in\Delta^{\circ}_H$ almost surely, $
\Delta^{\circ}_H$ as in \eqref{findimsimplexinterior}. 
Let
$W:=(W_1,\ldots ,W_H)$ be the WF diffusion with generator 
\eqref{eq:hatA0} and initial condition $W(0)$, $\beta_h$ be as in \eqref{eq:beta} for $
h=1,\ldots ,H$, 
and let $\Xt_1,\ldots ,\Xt_H$ be independent Poisson--Dirichlet 
diffusions with generators \eqref{IMNAoperator}, 
independent of $W$, each with parameter $\theta_h$ and with initial conditions 
$Z_h(0)/W_h(0)$ (note that 
\eqref{definitialcondition}-\eqref{requirementsoninitialcondition} 
and \eqref{eq:initconv} imply that $W(0)$, 
$Z_1(0)/W_1(0),\ldots ,Z_H(0)/W_H(0)$ are independent). Then 
\begin{equation}\nonumber
\rho^{K}_{(H)}\big(Z^K\big)
\overset{d}\longrightarrow
Z
\end{equation} 
where 
\[
Z_h(\cdot ):=W_h(\cdot )\mathcal{X}_h(\cdot ),\quad h=1,\ldots ,H,
\]
\begin{equation}
\label{eq:Zstar}\mathcal{
X}_h(t):=X_h\bigg(\int_0^t\beta_h(W(s))d {s}\bigg),
\quad h=1,\ldots ,H,
\end{equation}
with $\beta_h$ as in \eqref{eq:beta}. 
 $Z(t)\in\simplexHo$, with $\simplexHo$ as in 
\eqref{extendedKingmansimplexinterior}, for all $t>0$, 
almost surely. 
If $\rho^K_{(H)}\big(Z^K(0)\big)\overset{d}\longrightarrow Z(0)$ with $Z(0)\in\simplexHo$ almost surely, then \eqref{eq:initconv} holds with $W(0)=\big(\sum_{i\geq 1}Z_{1i}(0),\ldots ,\sum_{i\geq 
1}Z_{H,i}(0)\big)$ and $Z(t)\in\simplexHo$ for all $t\geq 0$, almost surely. 

\end{theorem}

\begin{proof}
Let $W^K,\Xt_1^K,\ldots ,\Xt_H^K$ be as in Theorem 
\ref{th:Krepr}.  By virtue of the Skorohod Representation Theorem, we can suppose 
that, almost surely, $\big(W^K,\rho^{K}(\Xt_1^K),\ldots ,\rho^{K}(\Xt_
H^K)\big)$ converges 
to $(W,\Xt_1,\ldots ,\Xt_H)$ uniformly over compact time 
intervals. Since, almost surely, $W(t)\in\Delta^{\circ}_H$ for all 
$t\geq 0$, it follows that $\rho^{K}\big(\Xt_h^K(\int_0^t\beta_h(W^K(s))d {s}
)\big)$ converges almost surely to 
$\Xt_h(\int_0^t\beta_h(W(s))d {s})$ uniformly over compact time intervals, for $
h=1,\ldots ,H$. The first assertion then follows from the fact that the map 
\begin{equation}
\begin{aligned}
\big(w(\cdot),x(\cdot)\big)\in C_{E}[0,\infty)\rightarrow S(w(\cdot),x(\cdot))\in C_{\simplexH}[0,\infty)
\end{aligned}
\end{equation}
is continuous with respect to uniform convergence on compact time intervals. 
The second assertion is an immediate consequence of the properties of Poisson-Dirichlet diffusions and the other assertions  follow by Lemma \ref{th:Sinv}. 
\end{proof}


\subsection{The generator of the limit in skew-product coordinates}\label{subsec:gen-on-E}

Let now $\Abb_{h}:=\Abb^{\theta_{h}}$, where $\Abb^{\theta_{h}}$ is as in 
\eqref{IMNAoperator} with $\theta$ replaced by $\theta_{h}$. Note that the algebra generated by the functions \eqref{eq:varphi} is also the linear space generated by the functions 
\begin{equation}\label{Phi grande}
\begin{aligned}
\Phi_{\mm^{h}}(z_{h}):=&\,\prod_{q\ge1}\varphi_{m^{h}_{q}}(z_{h}),\quad \mm^{h}=(m^{h}_{1},m^{h}_{2},m^{h}_{3},\ldots) \in {\cal M},\\
{\cal M}:=&\,\,\big\{{\mm}\in \N^\infty:\; m_1\geq m_2\geq \cdots,\, m_l=1 \text{  for } l\geq r+1 \text{ for some } r\in \N\big\},
\end{aligned}
\end{equation} 
(recall that $\varphi_1\equiv 1$). Thus, for every initial distribution on $\overline \nabla$, the Markov process with generator $\Abb_h$ with domain the algebra generated by the functions \eqref{eq:varphi} is also the unique solution of the martingale problem for $\Abb_h$ with domain 
\begin{equation}\label{D(Abb_h)}
\D(\Abb_h):=\bigg\{\Phi_{\mm^{h}},\,\mm^{h} \in {\cal M}\bigg\},
\end{equation}
where $\Phi_{\mm}$ and ${\cal M}$ are as in \eqref{Phi grande}. Therefore in the sequel we will suppose the domain of 
$\Abb_{h}$ to be given by \eqref{D(Abb_h)}.

For a generic metric space $G$, we denote by ${\cal C}_b(G)$ the space of bounded continuous functions on $G$. Define the operator $\Abb :\D(\Abb )\subseteq {\cal C}_b(E)\rightarrow 
{\cal C}_b(E)$ as 
\begin{equation}
\begin{aligned}
\label{eq:Astar}\Abb f(w,x):=\Amv_0f(w,x)+\sum_{h=1}^H\beta_h(w)
\Abb_{h} f(w,x),\end{aligned}
\end{equation} 
where, with a slight abuse of notation, we still denote by  $\Amv_0$ the operator that acts on $f$ as a function of $w$ like $A_{0}$ in \eqref{eq:hatA0} and, for each $
h=1,\ldots ,H$, by 
$\Abb_{h}$ the operator that acts on $f$ as a function of $x_h$ like $\Abb_{h}$. Let also 
\begin{equation}
\begin{aligned}
\label{eq:Astardomain}
\D(\Abb ):=\bigg\{&f=\prod_{h=0}^Hf_h,\,f_0\in\D(\Asv_
0),\,f_h\in\D(\Abb_{h}),\\
&\,\sup_{(w,x)\in E}f_0(w)\beta_h(w)|\Abb_{h} f_h(x_h)|<\infty,\, h=1,\ldots ,H\bigg\},
\end{aligned}
\end{equation} 
with $\D(\Asv_0)$ as in \eqref{eq:DA0}.  

\begin{theorem}\label{th:Astarmp}
Let $W$ be a WF diffusion with generator 
\eqref{eq:hatA0} and initial condition almost surely in $\Delta^{\circ}_H$, $\Xt_1,\ldots ,\Xt_H$ be independent Poisson--Dirichlet 
diffusions with generators \eqref{IMNAoperator}, each with parameter $\theta_h$, $h=1,\ldots,H$, 
independent of $W$, and let $(\mathcal{X}_1,\ldots ,\mathcal{X}_H)$ be given by \eqref{eq:Zstar}. 
Then $(W,\mathcal{X}_1,\ldots ,\mathcal{X}_H)$ is the unique 
(in distribution) solution of the martingale 
problem for $\Abb$ as in \eqref{eq:Astar}-\eqref{eq:Astardomain} 
with initial distribution given by the law of $\big(W(0),\Xt _1(0),\ldots,\Xt _H(0)\big)$.
\end{theorem}

\begin{proof}
The assertion follows by the same arguments used in 
the proofs of Theorems \ref{th:AKmpex} and 
\ref{th:AKunico}. 
\end{proof}

\begin{corollary}\label{th:mixt}
For every $\mu\in {\cal P}(E)$, the martingale problem for $\Abb$,  $\Abb$ as in \eqref{eq:Astar}-\eqref{eq:Astardomain}, with initial distribution $\mu$, has one and only one solution. The law of the solution can be represented as \begin{equation}\label{eq:mixt}
\int_{E}P_{(w,x)}\, d\mu ,
\end{equation}
where $P_{(w,x)}$ is the law of the solution of the martingale problem for $\Abb$ with initial distribution $\delta_{(w,x)}$.
\end{corollary}

\begin{proof}
The proof is analogous to that of Corollary \ref{th:Kmixt}. 
\end{proof}

The following Corollary extends the result of Theorem \ref{Z1,Z2aslimit} to the case when the initial condition $Z^K(0)$ does not satisfy \eqref{requirementsoninitialcondition}. For $\nu\in {\cal P}(\simplexH)$ such that $\nu(\simplexHo)=1$, define $\nu\circ S$ by setting, for every Borel subset $C$ of $E$, 
\begin{equation}\nonumber
(\nu\circ S)(C):=\nu (S(C\cap E^{\circ}))=\nu (S(C)\cap \simplexHo),
\end{equation}
(note that $S(C\cap E^{\circ})$ is a Borel set of $\simplexH$ by Lemma \ref{th:Sinv}). 

\begin{corollary}\label{th:geninit}
Let $Z^K$ be the unique (in distribution) solution of the martingale 
problem for $\B^K$ as in \eqref{eq:BK} with initial distribution $\nu^K\in {\cal P}\big(\Delta_{HK}\big)$ such that 
$\nu^K\big(\Delta^{H,\circ}_{HK}\big)=1$. Suppose that, as $K\rightarrow\infty$, $\big\{\nu^K\circ \big(\rho^K_{(H)}\big)^{-1}\big\}$ converges weakly to a probability distribution $\nu\in {\cal P}(\simplexH)$ such that $\nu(\simplexHo)=1$. Then 
\begin{equation}\nonumber
\rho^{K}_{(H)}\big(Z^K\big)
\overset{d}\longrightarrow
Z:=\big(W_1(\cdot )\,\mathcal{X}_1(\cdot ),\ldots ,W_H(\cdot 
)\,\mathcal{X}_H(\cdot )\big)
\end{equation} 
where $\big(W,\mathcal{X}\big)$ is the unique (in distribution) solution of the martingale problem for $\Abb$ with initial distribution $\nu\circ S$. 
$Z(t)\in\simplexHo$ for all $t\geq 0$, almost surely. 

\end{corollary}

\begin{proof}
By the second part of Theorem  \ref{th:Krepr}, $Z^K\stackrel d{=}S\big(W^K,\mathcal{X}^K\big)$, where $\big(W^K,\mathcal{X}^K\big)$ is the solution of the martingale problem for 
$\A^{K}$ as in \eqref{eq:AK}-\eqref{eq:DK} with initial distribution $\mu^K:=\nu^K\circ S$. Using Proposition 2.4 of \cite{EK81} and the same arguments as in Theorem \ref{th:AKmpex}, Theorem \ref{th:AKunico} and Corollary \ref{th:Kmixt} one can see that 
$\rho^K_{(H)}\big(W^K,\mathcal{X}^K\big)$ is the unique (in distribution) solution of the martingale problem for 
$\A^{K,\rho}$ defined by 
\begin{equation}
\big(\A^{K,\rho} f\big)\circ \rho^K_{(H)}:=\A^{K}\big(f\circ \rho^K_{(H)}\big),
\end{equation}
(note that the definition is well posed because $\A^{K}\big(f\circ \rho^K_{(H)}\big)$ takes a single value on $\big(\rho^K_{(H)}\big)^{-1}\big(\{(w,x)\}\big)$, for every $(w,x)\in \rho^K_{(H)}(E_K)$),
\begin{equation}
\begin{aligned}
\D(\A^{K,\rho}):=&\,\bigg\{f\in {\cal C}_b\big(\rho^K_{(H)}(E_K)\big):\,f\circ \rho^K_{(H)}=\tilde {f}\big\vert_{E_K},\,\tilde{
f}\in C^2\big(\overline {E_K}\big),\\
&\qquad\sup_{(w,x)\in E_K}\beta_h(w_
h)|\AAA_{h}^K\tilde{f}(w,x)|<\infty ,\,h=1,\ldots ,H\bigg\}\\
=&\,\bigg\{f\in {\cal C}_b\big(\rho^K_{(H)}(E_K)\big):\,f\circ \rho^K_{(H)}\in \D(\A^K)\bigg\},
\end{aligned}
\end{equation}
with initial distribution $\mu^K\circ \big(\rho^K_{(H)}\big)^{-1}$ and that the law of $\rho^K_{(H)}\big(W^K,\mathcal{X}^K\big)$ has the representation 
\begin{equation}\label{eq:Krhomixt}
\int_{\rho^K_{(H)}(E_K)}P^{K,\rho}_{(w,x)}\, d\big(\mu^K\circ \big(\rho^K_{(H)}\big)^{-1}\big) ,
\end{equation}
where $P^{K,\rho}_{(w,x)}$ is the law of the solution of the martingale problem for $\A^{K,\rho}$ with initial distribution 
$\delta_{(w,x)}$. 

By \eqref{eq:S-1comm}, for every Borel subset $C$ of $E$, 
\begin{equation}\nonumber
S\big(\big(\rho^K_{(H)}\big)^{-1}(C)\big)=S\big(\big(\rho^K_{(H)}\big)^{-1}(C\cap E^{\circ})\big)=\big(\rho^K_{(H)}\big)^{-1}\big(S\big(C\cap E^{\circ}\big)\big).
\end{equation}
Then $\mu^K\circ \big(\rho^K_{(H)}\big)^{-1}=\nu^K\circ \big(\rho^K_{(H)}\big)^{-1}\circ S$. Therefore, by the assumptions and the continuity of $S^{-1}$ as a function on on $\simplexHo$ (Lemma \ref{th:Sinv}),
\begin{equation}\nonumber
\mu^K\circ \big(\rho^K_{(H)}\big)^{-1}\overset{w}\longrightarrow \mu:=\nu\circ S,
\end{equation}
where $\overset{w}\longrightarrow$ denotes weak convergence. Let $(W,{\cal X})$ be the solution of the martingale problem for $\Abb$,  $\Abb$ as in \eqref{eq:Astar}-\eqref{eq:Astardomain}, with initial distribution $\mu$. We know, by the argument in the proof of Theorem \ref{Z1,Z2aslimit} that, for a sequence $\big\{(w^K,x^K)\big\}$ such that, for each $K$, $(w^K,x^K)\in \rho^K_{(H)}(E_K)$ and $(w^K,x^K)\rightarrow (w,x)\in E$, $P^{K,\rho}_{(w^K,x^K)}\overset{w}\longrightarrow P_{(w,x)}$, where $P_{(w,x)}$ is the law of the solution of the martingale problem for $\Abb$ with initial distribution $\delta_{(w,x)}$. Therefore, taking into account that $P_{(w,x)}$ is continuous in $(w,x)$ with respect to weak convergence, if we extend $P^{K,\rho}_{(w,x)}$ to $E$ by setting, for each $K$, 
\begin{equation}
P^{K,\rho}_{(w,x)}:=P_{(w,x)},\quad (w,x)\in E\setminus \rho^K_{(H)}(E_K),
\end{equation}
we have that $P^{K,\rho}_{(w,x)}\overset{w}\longrightarrow P_{(w,x)}$ uniformly over compact subsets of $E$. Since  $\big\{\mu^K\circ \big(\rho^K_{(H)}\big)^{-1}\big\}$ is tight, 
\begin{equation}\label{eq:Krhomixt}
\int_{\rho^K_{(H)}(E_K)}P^{K,\rho}_{(w,x)}\, d\big(\mu^K\circ \big(\rho^K_{(H)}\big)^{-1}\big)=
\int_{E}P^{K,\rho}_{(w,x)}\, d\big(\mu^K\circ \big(\rho^K_{(H)}\big)^{-1}\big)\overset{w}\longrightarrow \int_{E}P_{(w,x)}\, d\mu ,
\end{equation}
that is the law of $\rho^K_{(H)}\big(W^K,\mathcal{X}^K\big)$ converges weakly to the law of $\big(W,\mathcal{X}\big)$. 

The assertion then follows from Theorem \ref{th:Krepr} as in Theorem \ref{Z1,Z2aslimit}.

\end{proof}


\subsection{The generator of the limit in generalized Kingman simplex coordinates}\label{subsec:gen-on-simplex}

The following result identifies the generator of the process $Z$ in Theorem \ref{Z1,Z2aslimit} or, more generally, in Corollary \ref{th:geninit} as a composition of the operator \eqref{eq:Astar}-\eqref{eq:Astardomain} with the map $S^{-1}$ and shows that it is the limit of the operators $B^K$ in \eqref{eq:BK}, after the appropriate reordering of the test function arguments. 
\begin{theorem}\label{thm:Z and B implicit}
Let $Z$ be the limiting process 
in Theorem \ref{Z1,Z2aslimit} or, more generally, Corollary \ref{th:geninit}, with initial condition 
$Z(0)\in\simplexHo$ almost surely, $\simplexHo$ as in \eqref{extendedKingmansimplexinterior}. 
Then $Z$ is the unique (in distribution) solution of the martingale problem for the 
operator $\BB:{\cal D}(\BB)\subseteq {\cal C}_b(\simplexHo)\rightarrow 
{\cal C}_b(\simplexHo)$ defined 
as 
\begin{equation}
\begin{aligned}
\label{BwrittenasAstar}
\BB f(z):=&\,\Abb g(S^{-1}(z)),\quad z\in\simplexHo,\\
\D(\BB):=&\,\left\{f\in {\cal C}_b(\simplexHo):\;f\circ S=g\big\vert_{E^{\circ}},\,g\in\D(\Abb )\right\},
\end{aligned}
\end{equation}
where $\Abb$ is as in \eqref{eq:Astar}-\eqref{eq:Astardomain} and the definition of  $\BB$ is well posed because the density of $E^\circ$ in $E$ implies that, for every $f\in {\cal C}_b(\simplexHo)$, there is at most one $g\in \D(\Abb )$ such that $f\circ S=g\big\vert_{E^{\circ}}$. 

Moreover, for every $f\in \D(\BB)$, $f\circ \rho^K_{(H)}$ can be extended to a function on $\R^{HK}$ that has continuous  second derivatives at each $z\in \Delta^{H,\circ}_{HK}$ and, with $B^K$ as in \eqref{eq:BK}, 
\begin{equation}\label{B^K to B}
\sup_{z\in \Delta^{H,\circ}_{HK}}
\Big|\B^{K}(f\circ \rho^{K}_{(H)})(z)-\BB f(\rho^{K}_{(H)}(z))\Big|\rightarrow 0,
\end{equation} 
as $K\rightarrow \infty$. 
\end{theorem}

\begin{proof}
The assertions on the process $Z$ follow from Theorem \ref{th:Astarmp} and Corollary \ref{th:mixt}, together with Lemma~\ref{th:Sinv}. 

Let us now turn to \eqref{B^K to B}. By a slight abuse of notation, in the sequel we denote by $f\circ S$ the function in $ \D(\Abb)$ that is the unique continuous extension of $f\circ S$ to $E$ ($S$ as in \eqref{eq:Smap}). 
From \eqref{BwrittenasAstar}, we have
\begin{equation}\nonumber
\BB f(z)
=\Abb(f\circ S)(S^{-1}(z)), \quad \forall z\in \simplexHo. 
\end{equation} 
On the other hand, for $f\in \D(\BB)$, $f\circ \rho^K_{(H)}$ can be extended to a function on $\R^{HK}$ that has continuous first and second derivatives at each $z\in \Delta^{H,\circ}_{HK}$ and, by Lemma \ref{prop:generator mapping}, denoting by $B^K$ the expression in \eqref{eq:BK}, 
\begin{equation}\nonumber
\B^{K}(f\circ \rho^{K}_{(H)})(S(w,x))
=
\A^{K}(f\circ \rho^{K}_{(H)}\circ S)(w,x), \quad  \forall (w,x)\in E_{
K},
\end{equation} 
or equivalently, taking into account that $S$ is bijective between $E_K$ and $\Delta_{HK}^{H,\circ}$,
\begin{equation}\nonumber
\B^{K}(f\circ \rho^{K}_{(H)})(z)
=
\A^{K}(f\circ \rho^{K}_{(H)}\circ S)(S^{-1}(z)), 
\quad  \forall z\in \Delta^{H,\circ}_{HK}.
\end{equation} 
and hence, by \eqref{eq:Scomm}, 
\begin{equation}\nonumber
\B^{K}(f\circ \rho^{K}_{(H)})(z)
=\A^{K}(f\circ S \circ\rho^{K}_{(H)})(S^{-1}(z)).
\end{equation} 
Then, by \eqref{eq:S-1comm}, for all $z\in \Delta^{H,\circ}_{HK}$,
\begin{equation}\nonumber
\begin{aligned}
\Big|\B^{K}&(f\circ \rho^{K}_{(H)})(z)-\BB f(\rho^{K}_{(H)}(z))\Big|\\
=&\,
\Big|\A^{K}(f\circ S \circ \rho^{K}_{(H)})(S^{-1}(z))-\Abb(f\circ S)(S^{-1}(\rho^{K}_{(H)}(z)))\Big|\\
=&\,
\Big|\A^{K}((f\circ S) \circ \rho^{K}_{(H)})(S^{-1}(z))-\Abb(f\circ S)(\rho^{K}_{(H)}(S^{-1}(z)))\Big|
\end{aligned}
\end{equation} 
from which
\begin{equation}\nonumber
\begin{aligned}
\sup_{z\in \Delta^{H,\circ}_{HK}}&\Big|\B^{K}(f\circ \rho^{K}_{(H)})(z)-\BB f(\rho^{K}_{(H)}(z))\Big|\\
= &\,
\sup_{z\in \Delta^{H,\circ}_{HK}}
\Big|\A^{K}((f\circ S) \circ \rho^{K}_{(H)})(S^{-1}(z))-\Abb(f\circ S)(\rho^{K}_{(H)}(S^{-1}(z)))\Big|\\
= 
&\,
\sup_{(w,x)\in E_{K}}
\Big|\A^{K}((f\circ S) \circ \rho^{K}_{(H)})(w,x)-\Abb(f\circ S)(\rho^{K}_{(H)}(w,x))\Big|
\end{aligned}
\end{equation} 
which goes to zero as $K\rightarrow \infty$ by $(2.15)$ in \cite{EK81}, taking into account that, for $g:=f\circ S \in \D(\Abb)$, $\sup_{\Delta_H^{\circ}}|g_0(w)\beta_h(w)|\le 1$ for all $h$ such that $\Abb_h g_h$ is not identically $0$.
\end{proof}

The next theorem, which concludes this section, provides an explicit description of $\BB$ and $\D(\BB)$ in Theorem \ref{thm:Z and B implicit}. 

\begin{theorem}\label{explicit B and domain}
Let $\BB$ and $\D(\BB)$ be as in Theorem \ref{thm:Z and B implicit}. Then
\begin{equation}\label{final D(B)}
\begin{aligned}
\D(\BB)=\bigg\{\prod_{h=1}^{H}
|z_{h}|^{m^{h}_{0}}\Phi_{\mm^{h}}(z_{h}),\, 
{\mm^h}\in {\cal M},\, & m^{h}_{0}\geq 1-\sum_{q\ge1}m^{h}_{q}\ind_{m^{h}_{q}\ge2},\,\mbox{\rm for}\, h\, \mbox{\rm s.t.}\,\Phi_{\mm^{h}}\, \mbox{\rm is not identically}\, 1,\\ 
& m^{h}_{0}\geq 0,\,\mbox{\rm for}\, h\, \mbox{\rm s.t.}\,\Phi_{\mm^{h}}\, \mbox{\rm is identically}\, 1\bigg\},
\end{aligned}
\end{equation} 
where $|z_{h}|:=\sum_{i\ge1}z_{h,i}$, and $\Phi_{\mm}$, ${\cal M}$ are given by \eqref{Phi grande}, and 
\begin{equation}\label{final B}
\begin{aligned}
\BB f(z)=&\,\frac{1}{2}\sum_{h,k=1}^H\sum_{i,j=
1}^{\infty}z_{h,i}(\delta_{h,k}\delta_{i,j}-z_{k,j})\frac{\partial^
2f}{\partial z_{h,i}\partial z_{k,j}}\\
&\,-\sum_{h=1}^{H}\frac{\tn }{2}\sum_{i=1}^{\infty}z_{h,i}\frac{\partial f}{\partial z_{h,i}}
+\frac{1}{2}\sum_{h=1}^{H}\theta_{h}m^{h}_{0}|z_{h}|^{-1}f(z).
\end{aligned}
\end{equation} 
\end{theorem}
\begin{proof}
Since $S$ is bijective between $E^{\circ}$ and $\simplexHo$, a function $f$ belongs to $\D(\BB)$ in \eqref{BwrittenasAstar} if and only if 
\[f(z)=g(S^{-1}(z)),\;\forall z\in \simplexHo,\quad {\textrm {for some }}g\in \D(\Abb).\]
Since $g$ belongs to $\D(\Abb )$ if and only if 
\[g(w,x)=w_1^{p_{1}}\cdots w_H^{p_{H}}\prod_{h=1}^{H}\Phi_{\mm^{h}}(x_{h}),
\]
where $\mm^{h}\in {\cal M}$, $\Phi_{\mm}$ and ${\cal M}$ are given by \eqref{Phi grande} and $p_h\geq 1$ for each $h$ such that $\Phi_{\mm^{h}}$ is not identically $1$, $f$ belongs to $\D(\BB)$ if and only if it has the form 
\[f(z)=|z_1|^{p_{1}}\cdots |z_H|^{p_{H}}\prod_{h=1}^{H}\Phi_{\mm^{h}}\bigg(\frac{z_h}{|z_h|}\bigg)=|z_1|^{m^1_{0}}\cdots |z_H|^{m^H_{0}}\prod_{h=1}^{H}\Phi_{\mm^{h}}(z_h),
\]
with $m^{h}_{0}= p_h-\sum_{q\ge1}m^{h}_{q}\ind_{m^{h}_{q}\ge2},\, h=1,\ldots,H$, which gives \eqref{final D(B)}.

The explicit form of $\BB$ can be found by taking advantage of \eqref{B^K to B}. 
Let $\B^{K}$ be as in \eqref{eq:BK}, $\rho^{K}_{(H)}$ as in \eqref{rho_K vector operator},
and define, for $f\in \D(\BB)$, $\tilde \BB f$ as the right-hand side of \eqref{final B}. Using the fact that 
\begin{equation}\label{derivatives for n1,m1,m2}
\begin{aligned}
\frac{\partial}{\partial z_{h,i}}\bigg(|z_h|^{m^h_{0}}\Phi_{\mm^{h}}(z_{h})\bigg)
=&\,
m^{h}_{0}(|z_{h}|)^{m^{h}_{0}-1}\Phi_{\mm^{h}}(z_{h})\\
&\,+(|z_{h}|)^{m^{h}_{0}}
\sum_{q \ge 1}\ind_{m^{h}_{q}\ge2}\,
m^{h}_{q}z_{hi}^{m^{h}_{q}-1}\prod_{l\ne q}\varphi_{m^{h}_{l}}(z_{h}),
\end{aligned}
\end{equation}
and that $\rho^{K}_{(H)}(z)_i=0$ for $i\geq K+1$, we find, for $z\in \Delta^{H,\circ}_{HK}$, 
\begin{equation}\nonumber 
\begin{aligned}
\Big|\B^{K}&\,(f\circ \rho^{K}_{(H)})(z)-\tilde \BB f(\rho^{K}_{(H)}(z))\Big|
= \frac {1}{2}\sum_{h=1}^{H}
\Big| \sum_{i=1}^{K}\frac{\theta_{h}}{K}\frac{\partial f}{\partial z_{h,i}}(z)
-\theta_{h}m^{h}_{0}|z_{h}|^{-1}f(z) \Big|\\
=&\, \frac {1}{2}\sum_{h=1}^{H}\theta_{h}
\Bigg| \sum_{i=1}^{K}\frac{1}{K}\bigg(m^{h}_{0}|z_{h}|^{m^{h}_{0}-1}\Phi_{\mm^{h}}(z_{h})
+|z_{h}|^{m^{h}_{0}}\sum_{q\ge 1}\ind_{m^{h}_{q}\ge2}
m^{h}_{q}z_{h,i}^{m^{h}_{q}-1}\prod_{l\ne q}\varphi_{m^{h}_{l}}(z_{h})\bigg)\\
&\,\quad \quad \quad -m^{h}_{0}|z_{h}|^{m^{h}_{0}-1} \Phi_{\mm^{h}}(z_{h})\bigg|
\prod_{k\ne h}|z_{k}|^{m^{k}_{0}}\Phi_{\mm^{k}}(z_{k})\\
= &\,\frac{1}{2K}\sum_{h=1}^{H}\theta_{h}
|z_{h}|^{m^{h}_{0}}\sum_{q\ge 1}\ind_{m^{h}_{q}\ge2}\,
m^{h}_{q}\sum_{i=1}^{K}z_{h,i}^{m^{h}_{q}-1}\prod_{l\ne q}
\varphi_{m^{h}_{l}}(z_{h})
\prod_{k\ne h}|z_{k}|^{m^{k}_{0}}\Phi_{\mm^{k}}(z_{k}).
\end{aligned}
\end{equation} 
Now, for $m\ge 2$, 
\begin{equation}
\sum_{i=1}^{K}z_{h,i}^{m-1}\le |z_{h}|^{m-1},\qquad \varphi_{m}(z_{h})\le  |z_{h}|^{m},
\end{equation}
which implies, in particular, taking also into account that $\varphi_{1}\equiv 1$ and the definition of $\D(\BB)$, 
\begin{equation}
|z_{k}|^{m^{k}_{0}}\Phi_{\mm^{k}}(z_{k})\le |z_k|^{m^{k}_{0}+\sum_{l\ge 1}\ind_{m^{k}_{l}\ge 2}m^{k}_{l}}\le 1,\quad \forall k=1,\ldots,H.
\end{equation}
Therefore, keeping in mind the definition of $\D(\BB)$, 
\begin{equation}
\begin{aligned}
\Big|\B^{K}\,(f\circ \rho^{K}_{(H)})(z)-\tilde \BB f(\rho^{K}_{(H)}(z))\Big|&\le \frac{1}{2K}\sum_{h=1}^{H}\theta_{h}\sum_{q\ge 1}\ind_{m^{h}_{q}\ge2}\,
m^{h}_{q}|z_h|^{m^{h}_{0}+\sum_{l:\,m^{h}_{l}\ge 2}m^{h}_{l}-1}\\
&\le \frac{1}{2K}\sum_{h=1}^{H}\theta_{h}\sum_{q\ge 1}\ind_{m^{h}_{q}\ge2}\,
m^{h}_{q},
\end{aligned}
\end{equation} 
Since $\theta_{h}$ and $m^{h}_{q}$, for all $h,q$, are fixed, and $\mm^{h}\in {\cal M}$ has at most finitely-many coordinates that are at least 2, the previous implies 
\begin{equation}\label{conv B^K to B tilde}
\sup_{z\in \Delta^{H,\circ}_{HK}}
\Big|\B^{K}(f\circ \rho^{K}_{(H)})(z)-\tilde\BB f(\rho^{K}_{(H)}(z))\Big|\rightarrow 0,
\quad \text{as $K\rightarrow \infty$},
\end{equation}
Moreover, for every $z\in \simplexHo$, there exists a sequence $\{z^{K}\}$ such that $z^{K}\in \Delta^{H,\circ}_{HK}$ for all $K\ge1$ and $\rho^{K}_{(H)}(z^{K})\rightarrow z$. For example, one can take $z^{K}=(z^{K}_{1},\ldots,z^{K}_{H})$ as 
\begin{equation}\nonumber
z^{K}_{h,1}:=z_{h,1}+\sum_{j\ge K+1}z_{h,j}, \quad 
z^{K}_{h,i}:=z_{h,i}, \quad 2\le i\le K.
\end{equation} 
Then, for every $z\in \simplexHo$, we can write 
\begin{equation}\label{B-B tilde triangle ineq}
\begin{aligned}
\Big|&\,\BB f(z)-\tilde \BB f(z)\Big|
\le
\Big|\BB f(z)-\BB f(\rho^{K}_{(H)}(z^{K}))\Big|\\
&\,+
\Big|\BB f(\rho^{K}_{(H)}(z^{K}))-\tilde \BB f(\rho^{K}_{(H)}(z^{K}))\Big|
+
\Big|\tilde \BB f(\rho^{K}_{(H)}(z^{K}))-\tilde \BB f(z)\Big|,
\end{aligned}
\end{equation} 
As $K\rightarrow \infty$, the first and the third terms on the right-hand side of \eqref{B-B tilde triangle ineq} converge to zero  because $\BB f$ and $\tilde \BB f$ are continuous on $\simplexHo$, while the second term goes to zero by \eqref{conv B^K to B tilde} and \eqref{B^K to B}.
Hence $\BB f(z)=\tilde \BB f(z)$ for all $z\in \simplexHo$.
\end{proof}

\begin{remark}\label{rem:comparison-BO09}
For $H=2$, consider the signed second power sum
\[
q_1(z)=\varphi_2(z_1)-\varphi_2(z_2)
=\sum_{i\ge1}z_{1,i}^2-\sum_{i\ge1}z_{2,i}^2.
\]
For $h=1,2$, the function $\varphi_2(z_h)$ belongs to
$\D(\BB)$, and \eqref{final B} yields
\[
\BB\varphi_2(z_h)
=
|z_h|-(1+\tn)\varphi_2(z_h).
\]
Extending $\BB$ to the linear space generated by $\D(\BB)$ by linearity, we have 
\[
\BB q_1(z)=|z_1|-|z_2|-(1+\tn)q_1(z).
\]
By contrast, equation (1.5) of \cite{O10}, with the
Jack parameter equal to $1$, gives
\[
\mathcal L_{\zeta,\zeta'}q_1
=
2(\zeta+\zeta')-2(1+\zeta\zeta')q_1,
\]
where $\zeta,\zeta'$ are the parameters of that model.
Since $|z_1|-|z_2|$ is not determined by $q_1$,
the two operators cannot coincide for any choice of
parameters or constant rescaling of time.
The parameter-one case of \cite{O10} recovers
the diffusions of \cite{BO09}.
\end{remark}

\subsection{Some remarks on the convergence argument and the limiting operator}\label{subsec:domain-subtlety}

Looking at \eqref{conv B^K to B tilde}, one might wonder whether one could obtain the desired convergence in a simpler way, by the classical results of \cite{EK86}, Chapter 1, Section 6. This is not the case for at least two reasons. First, operator-level approximations such as \eqref{conv B^K to B tilde} must be proved on $\D(\tilde{\BB})$ (or a core for $\tilde{\BB}$) and, in our setting, $\D(\tilde{\BB})$ can be identified only through the implicit characterization of Theorem~\ref{thm:Z and B implicit}. Second, even if \eqref{conv B^K to B tilde} is verified on some domain $\D(\tilde{\BB})$, this does not, by itself, prove convergence, because one would also need to know that $\tilde\BB$, with domain $\D(\tilde{\BB})$, generates a Markov semigroup.

Next, we wish to point out that, although $\simplexHo$ is dense in $\simplexH$ (cf.~\eqref{extendedKingmansimplex}), for some $f\in\D(\BB)$, even with the restriction that $f$ can be extended continuously to $\simplexH$, it may be impossible to extend continuously $\BB f$ to $\simplexH$, so that $\BB$ cannot be viewed as an operator on $\C(\simplexH)$ as was the case, for $H=1$, in \cite{EK81,P09,Cea17}. The obstruction is the discontinuity of the inverse map $S^{-1}$ at points of $\simplexH\setminus\simplexHo$. Example \ref{re:Bstar-ext} below shows that there exist sequences $\{z^n\}\subset \simplexHo$ converging to $z\in\simplexH\setminus\simplexHo$ for which $S^{-1}(z^n)$ has distinct subsequential limits, or, equivalently, the decomposition $z_h=w_hx_h$ is not stable in the product topology: the maps $z\mapsto |z_h|:=\sum_{i\ge1}z_{h,i}$ are not continuous at points of $\simplexH\setminus\simplexHo$, so the mark masses and within-mark normalizations may depend on the approximating sequence (this may happen because dust appears, i.e.\ $\sum_{h,i}z_{h,i}<1$ in the limit, or because some mark mass vanishes, or both). 
Since $\BB f(z):=\,\Abb g(S^{-1}(z))$, with $g$ as in \eqref{BwrittenasAstar}, and $\Abb g$ can be continuously extended to $\overline{E^{\circ}}$,  $\BB f(z^n)$ will also have distinct subsequential limits.

\begin{example}[Failure of continuity of $S^{-1}$ at $\simplexH\setminus\simplexHo$]
\label{re:Bstar-ext}
Let $H=2$ and $z^n=(z_1^n,z_2^n)\in\simplexHo$ be given by 
\begin{equation}\label{eq:zndef}
\begin{aligned}
z_{1i}^n:=&\,\left\{\begin{array}{ll}
\displaystyle \frac{1}{2^{i+2}}+[1+(-1)^n]\frac{1}{8n},&1\leq i\leq n,\\[0.4em]
\displaystyle \frac{1}{2^{i+2}},&i\geq n+1,\end{array}
\right.\\
z_{2i}^n:=&\,\left\{\begin{array}{ll}
\displaystyle \frac{1}{2^{i+1}}+[1+(-1)^{n+1}]\frac{1}{8n},&1\leq i\leq n,\\[0.4em]
\displaystyle \frac{1}{2^{i+1}},&i\geq n+1.\end{array}
\right.
\end{aligned}
\end{equation}
Then $z^n\to z$ 
coordinatewise, where $z_{1i}=2^{-(i+2)}$ and $z_{2i}=2^{-(i+1)}$, so that
\[
\sum_{i\ge1}z_{1i}=\frac14,\qquad \sum_{i\ge1}z_{2i}=\frac12,
\qquad \sum_{h=1}^2\sum_{i\ge1}z_{h,i}=\frac34<1,
\]
(and hence $z\notin\simplexHo$, i.e. dust is created in the limit).
On the other hand, the mark masses of $z^n$ satisfy
\[
|z_1^n|:=\sum_{i\ge1}z_{1i}^n\to
\left\{\begin{array}{ll}
\frac12,&n\ \mathrm{even},\\[0.15em]
\frac14,&n\ \mathrm{odd},
\end{array}\right.
\qquad
|z_2^n|:=\sum_{i\ge1}z_{2i}^n\to
\left\{\begin{array}{ll}
\frac12,&n\ \mathrm{even},\\[0.15em]
\frac34,&n\ \mathrm{odd},
\end{array}\right.
\]
because the perturbation adds total mass $1/4$ to mark $1$ when $n$ is even and to mark $2$ when $n$ is odd.
Writing $S^{-1}(z^n)=(w_1^n,w_2^n,x_1^n,x_2^n)$ with $w_h^n=|z_h^n|$ and $x_h^n=z_h^n/|z_h^n|$, it follows that $S^{-1}(z^n)$ has two distinct subsequential limits:
along odd $n$,
\begin{equation}\label{eq:oddlim}
(w_1^n,w_2^n,x_1^n,x_2^n)\to
(w_1,w_2,x_1,x_2),
\qquad
w_{1}=\frac {1}{4},\ \ w_{2}=\frac {3}{4},\ \ 
x_{1i}=\frac {1}{2^i},\ \ x_{2i}=\frac {2}{3}\frac {1}{2^i},
\end{equation}
whereas along even $n$,
\begin{equation}\label{eq:evenlim}
(w_1^n,w_2^n,x_1^n,x_2^n)\to
(\tilde w_1,\tilde w_2,\tilde x_1,\tilde x_2),
\qquad
\tilde w_{1}=\frac{1}{2},\ \ \tilde w_{2}=\frac{1}{2},\ \ 
\tilde x_{1i}=\frac {1}{2^{i+1}},\ \ \tilde x_{2i}=\frac {1}{2^i}.
\end{equation}
Now let
\begin{equation}\nonumber
f(z_1,z_2):=\varphi_2(z_1)\varphi_2(z_2),
\end{equation}
where $\varphi_2$ is given by \eqref{eq:varphi}. $f\in\D(\BB)$, and the corresponding $g$ in \eqref{BwrittenasAstar} is
\begin{equation}\nonumber
g(w_1,w_2,x_1,x_2):=w_1^2 w_2^2
\varphi_2 (x_1)\varphi_2 (x_2).
\end{equation}
$\Abb g$ is continuous on $\overline{E^{\circ}}=\Delta_H\times\simplex^H$. Therefore $\BB f(z^n)=\Abb g(S^{-1}(z^n))$ has two subsequential limits,
namely $\Abb g(w_1,w_2,x_1,x_2)$,
and $\Abb g(\tilde w_1,\tilde w_2,\tilde x_1,\tilde x_2)$, 
which are different. 
\end{example}


\section{The multiple Poisson--Dirichlet distribution: stationarity and approximation }\label{sec:ALTstationarity}

In this section we show that the process constructed in Theorem \ref{Z1,Z2aslimit} is stationary with respect to the multiple Poisson--Dirichlet distribution introduced by \cite{S24a,S24b}.
We then conclude with a construction of the multiple Poisson--Dirichlet distribution $\PD
(\theta_1,\ldots,\theta_H)$ 
in the spirit of Kingman's construction of the $\PD(\theta )$ distribution through limits of ranked Dirichlet distributed random frequencies. 

It is often convenient to view the Poisson--Dirichlet distribution  as a probability measure on $\simplex$ by canonical extension  and we do so in the sequel. 

For $E$ as in \eqref{eq:E}, define the distribution $\mu^{*}$ on $
E$ as
\begin{equation}\label{ALTmustar}
\begin{aligned}
\mu^{*}:=\bigtimes_{h=0}^H\,\mu^{*}_h,\quad\mu^{*}_
0:=\mathrm{Dir}_H(\theta_1,\ldots ,\theta_H)\big\vert_{\Delta_{H}^{\circ}},\quad\mu^{*}_h:=\text{
PD}(\theta_h),\quad h=1,\cdots,H,\end{aligned}
\end{equation} 
where $\PD(\theta_h)$ denotes the Kingman's Poisson--Dirichlet distribution with parameter $
\theta_h$. Note that 
\begin{equation}\label{eq:ALTmustar0}
\mu^{*}(E\setminus E^{\circ})=0.
\end{equation}
The next result shows that $\mu^{*}$ is an invariant measure for $
\Abb$ as in \eqref{eq:Astar}.

\begin{theorem}\label{prop:ALTAstar-Ech}
The solution of the martingale problem for $\Abb$ as in  \eqref{eq:Astar}, with initial distribution $\mu^{*}$, is a 
stationary process. 
\end{theorem}

\begin{proof}
By Theorem \ref{th:Astarmp}, the solution of the martingale problem for $\Abb$ with initial distribution $\mu^{*}$ is 
$(W,\mathcal{X}_1,\ldots ,\mathcal{X}_H)$, where $W$ is the Wright--Fisher diffusion with generator \eqref{eq:hatA0} and initial distribution $\mu^{*}_0$, $\mathcal{X}_h$ is given by \eqref{eq:Zstar} with $X_h$ a Poisson--Dirichlet diffusion with initial distribution $\mu^{*}_h$ and $W,X_1,\ldots,X_H$ are independent. Since $\mu^{*}_0$ is the stationary distribution of $W$, $\bigtimes_{h=1}^H\,\mu^{*}_h$ is the stationary distribution of $(X_1,\ldots,X_H)$ and $(X_1,\ldots,X_H)$ is independent of $W$,  by conditioning on $W$ we see that the distribution of $(W,\mathcal{X}_1,\ldots ,\mathcal{X}_H)(t)$ is $\bigtimes_{h=0}^H\,\mu^{*}_h$ for each $t>0$. Then the assertion follows from the Markov property of $(W,\mathcal{X}_1,\ldots ,\mathcal{X}_H)$ (Corollary \ref{th:mixt}). 
\end{proof}

The following Corollary of Theorem \ref{prop:ALTAstar-Ech} identifies 
a stationary distribution of the multiple Poisson--Dirichlet diffusion $Z$ defined in Theorem \ref{Z1,Z2aslimit}. Let 
\begin{equation}\label{eq:nustar}
\nu^{*}:=\mu^{*}\circ S^{-1},
\end{equation}
where $\mu^{*}$ is given by \eqref{ALTmustar}. Then $\nu^{*}$ is the canonical extension to $\simplexH$ of the multiple Poisson--Dirichlet distribution of \cite{S24a,S24b}, cf. Definition \ref{def:mPD}. \eqref{eq:S-1simplexHo} and \eqref{eq:ALTmustar0} imply 
\begin{equation}\label{eq:ALTnustar0}
\nu^{*}(\simplexH\setminus\simplexHo)=0.
\end{equation}
With a slight abuse of notation, we still denote by $\nu^{*}$ the restriction of $\mu^{*}\circ S^{-1}$ to $\simplexHo$. 

\begin{corollary}\label{th:ALTBstar-stat}
Let $Z$ be the process defined in Theorem \ref{Z1,Z2aslimit} with initial 
distribution $\nu^{*}:=\big(\mu^{*}\circ S^{-1}\big)\big{|}_{\simplexHo}$.Then $Z$ is a stationary process. 
\end{corollary}
\begin{proof}
Let $(W,\Xtau):=(W,\mathcal{X}_1,\ldots ,\mathcal{X}_H)$  be as in Theorem \ref{Z1,Z2aslimit}, hence, in particular, $S(W,\Xtau)(0)=Z(0)$.
\eqref{eq:S-1simplexHo}, 
the fact that $S$ is bijective between $E^{\circ}$ and $\simplexHo$, \eqref{eq:ALTmustar0} and \eqref{eq:ALTnustar0} yield, for every Borel subset $B$ of $E$, 
\begin{equation}\nonumber
\begin{aligned}
\PP\big((W,\Xtau)(0)\in B\big)
=&\,\PP\big((W,\Xtau)(0)\in B
,\, Z(0)\in \simplexHo)\\
=&\,\PP\big((W,\Xtau)(0)\in B\cap E^{\circ},\, Z(0)\in \simplexHo\big)\\
=&\,\PP\big(S^{-1}(Z(0))\in B\cap E^{\circ},\, Z(0)\in \simplexHo\big)\\
=&\,\PP\big(Z(0)\in S(B\cap E^{\circ}),\, Z(0)\in \simplexHo\big)\\
=&\,\PP\big(Z(0)\in S(B\cap E^{\circ})\big)=\nu^{*}\big(S((B\cap E^{\circ}))\big)\\
=&\,\mu^{*}\big(S^{-1}(S(B\cap E^{\circ}))\big)=\mu^{*}\big(B\cap E^{\circ}\big)=\mu^{*}(B),
\end{aligned}
\end{equation} 
that is the initial distribution of $(W,\mathcal{X}_1,\ldots ,\mathcal{X}_H)$ is $\mu^{*}$. The assertion then follows from Theorem \ref{prop:ALTAstar-Ech}. 
\end{proof}

We have thus constructed a diffusion process for which the multiple Poisson--Dirichlet distribution is a stationary distribution.

Just as the multiple Poisson--Dirichlet diffusion $Z$ is the limit of blockwise ranked WF diffusions (Theorem \ref{Z1,Z2aslimit}), the multiple Poisson--Dirichlet distribution can be approximated by the distributions of  blockwise ranked Dirichlet random frequencies.
Recall that $\rho^K$ is defined by \eqref{eq:rho} and $\rho^K_{(H)}$ is defined by \eqref{rho_K vector operator} and \eqref{rho_K vector operator E_K}.

\begin{theorem} \label{th:MPDappr}
Let $\nu^{*K}$ be given by 
\begin{equation}\nonumber\label{eq:ALTDirHK}\nu^{*K}:=\mathrm{Dir}_{HK}\bigg(\frac{
\theta_1}K,\ldots ,\frac{\theta_1}K,\frac{\theta_2}K,\ldots ,\frac{
\theta_2}K,\ldots ,\frac{\theta_H}K,\ldots ,\frac{\theta_H}K\bigg
)\bigg\vert_{\Delta_{HK}^{H,\circ}}\end{equation}
 where each $\theta_h$ is repeated $K$ times, and let $\nu^{*}$ be given by \eqref{eq:nustar}. Then 
\[\nu^{*K}\circ\big(\rho^K_{(H)}\big)^{-1}\overset{w}\longrightarrow\nu^{
*}.\]
\end{theorem}
\begin{proof}
Let $\mu^{*K}$ be the probability measure on $E_K$ ($E_K$ as in \eqref{eq:EK}) given by 
\begin{equation}\nonumber
\begin{aligned}
\label{ALTBetaxDirxDir}\mu^{*K}:=\bigtimes_{h=0}^H\,\mu_h^{*K},\quad
\mu_0^{*K}=\mu^{*}_0:=\mathrm{Dir}_H(\theta_1,\ldots ,\theta_H)\big\vert_{\Delta_{H}^{\circ}},\quad
\mu_h^{*K}:=\mathrm{Dir}_K\bigg(\frac{\theta_h}K,\ldots ,\frac{\theta_
h}K\bigg).\end{aligned}
\end{equation}
By Proposition \ref{pro:Dirselfsimilarity}, 
\begin{equation}
\begin{aligned}
\label{eq:nuH^K}\nu^{*K}=\mu^{*K}\circ S^{-1}.\end{aligned}
\end{equation} 
Then, by \eqref{eq:Scomm}, we have 
\[\nu^{*K}\circ\big(\rho^K_{(H)}\big)^{-1}=\mu^{*K}\circ S^{-1}\circ\big
(\rho^K_{(H)}\big)^{-1}=\mu^{*K}\circ\big(\rho^K_{(H)}\big)^{-1}\circ S^{-1}
,\]
where in the first equality $S:E_K\rightarrow\Delta_{HK}^{H,\circ}$, with $
E_K$ and $\Delta_{HK}^{H,\circ}$ 
given by \eqref{eq:EK} and \eqref{eq:DeltaK2} respectively, and in the second equality $S:E\rightarrow\simplexH$, with $E$ and $\simplexH$ given 
by \eqref{eq:E} and \eqref{extendedKingmansimplex}
respectively. 
From \cite{K75} (cf.~also Corollary 4.7 in \cite{EK81} for a more general result) we have that 
$\mu^{*K}_h\circ(\rho^{K})^{-1}\overset{w}\longrightarrow\mu^{*}_h$, where $
\mu^{*}_h$ is a $\PD(\theta_h)$ distribution as in \eqref{ALTmustar}. 
Therefore 
\[\mu^{*K}\circ\big(\rho^K_{(H)}\big)^{-1}\overset{w}\longrightarrow\mu^{*}.\]
Since $\nu^{*} =\mu^{*}\circ S^{-1}$ and $S:E\rightarrow\simplexH$ is continuous, the assertion 
follows. 
\end{proof}





\section*{Acknowledgements}
We are grateful to an Associate Editor and a Referee for carefully reading our manuscript and for providing constructive suggestions. 
MR was supported by the European Union -- NextGenerationEU,
   PRIN 2022 PNRR, Grant no. P2022H5WZ9.


\end{document}